\documentclass[smallextended]{svjour3} 
\usepackage{enumerate}
\usepackage{mathrsfs}
\usepackage{graphicx}
\usepackage{pstricks}
\usepackage{pst-plot}
\usepackage{pstricks-add}
\usepackage{amsmath}
\usepackage{amssymb}
\usepackage{fancybox}
\newcommand{\boxedtxt}[1]{%
    \[\fbox{%
        \addtolength{\linewidth}{-2\fboxsep}%
        \addtolength{\linewidth}{-2\fboxrule}%
        \begin{minipage}{\linewidth}%
        #1 
        \end{minipage}%
      }\]%
  }
\newcommand{\Ucal}{\ensuremath{\mathcal U}}

\renewcommand{\Bbb}{\ensuremath{\mathbb B}}

\newcommand{\Nbb}{\ensuremath{\mathbb N}}

\newcommand{\Rbb}{\ensuremath{\mathbb R}}

\newcommand{\ghat}{{\widehat{g}}}

\newcommand{\xhat}{{\widehat{x}}}

\newcommand{\Tbar}{{\overline{T}}}

\newcommand{\mbar}{{\overline{m}}}

\newcommand{\xbar}{{\overline{x}}}
\newcommand{\ybar}{{\overline{y}}}

\newcommand{\epsilonbar}{{\overline{\epsilon}}}
\newcommand{\varepsilonbar}{{\overline{\varepsilon}}}

\newcommand{\lambdabar}{{\overline{\lambda}}}

\newcommand{\DRLRn}{{\Rbb^n}}

\newcommand{\DRLRp}{{\Rbb_+}}

\newcommand{\DRLset}[2]{\left\{#1\,\left|\,#2\right.\right\}}
\newcommand{\DRLmap}[3]{#1:\,#2\rightarrow #3\,}
\newcommand{\DRLmmap}[3]{#1:\,#2\rightrightarrows #3\,}

\newcommand{\DRLip}[2]{\left\langle #1,~ #2\right\rangle}

\newcommand{\und}{\quad\mbox{ and }\quad}

\newcommand{\bdy}{\partial}

\DeclareMathOperator{\argmax}{argmax\,}
\DeclareMathOperator{\argmin}{argmin\,}

\DeclareMathOperator{\DRLrange}{{\rm range}}

\DeclareMathOperator{\DRLdom}{{dom\,}}
\DeclareMathOperator{\gph}{gph}

\DeclareMathOperator{\rint}{ri\,}
\DeclareMathOperator{\intr}{int\,}

\DeclareMathOperator{\sol}{SOL}
\DeclareMathOperator{\vi}{VI}
\DeclareMathOperator{\gvi}{GVI}
\DeclareMathOperator{\mgvi}{MGVI}
\DeclareMathOperator{\wgvi}{WGVI}

\newcommand{\DRLsd}{\partial}

\usepackage{theorem} %
\usepackage[vlined, boxed]{algorithm2e}
\newtheorem{thm}{Theorem}
\newtheorem{cor}[thm]{Corollary}
\newtheorem{propn}[thm]{Proposition}
\newtheorem{defn}[thm]{Definition}
\newtheorem{assumption}[thm]{Assumption}

\theoremstyle{plain}{\theorembodyfont{\rmfamily}%

\theoremstyle{plain}{\theorembodyfont{\rmfamily}%
\theoremstyle{plain}{\theorembodyfont{\rmfamily}%
\newtheorem{eg}[thm]{Example}
\theoremstyle{plain}{\theorembodyfont{\rmfamily}%
\def\proof{\noindent{\it Proof}. \ignorespaces}
\def\endproof{\vbox{\hrule height0.6pt\hbox{\vrule height1.3ex%
width0.6pt\hskip0.8ex\vrule width0.6pt}\hrule height0.6pt}\medskip}
\newcommand{\SGeps}{S_{G_{\varepsilon\varphi}}}
\renewcommand{\epsilon}{\varepsilon}
\title{Lagrange Multipliers, (Exact) Regularization and Error Bounds for
Monotone Variational Inequalities}
\author{
  C. Charitha\thanks{
Institut f\"ur Numerische und Angewandte Mathematik,
Universit\"at G\"ottingen,
37083 G\"ottingen, Germany.  CC was supported by DFG Grant SFB755-A4. E-mail: \texttt{c.cherugondi@math.uni-goettingen.de}
},~
Joydeep Dutta\thanks{Department of Humanities and Social Sciences,
Indian Institute of Technology Kanpur
208016 Kanpur, India. Email: \texttt{jdutta@iitk.ac.in} }, and
~D.~ Russell Luke\thanks{
Institut f\"ur Numerische und Angewandte Mathematik,
Universit\"at G\"ottingen,
37083 G\"ottingen, Germany.  DRL was supported by DFG Grant SFB755-C2. Email: \texttt{r.luke@math.uni-goettingen.de} 
}  
}

\begin{document}
\maketitle

\begin{abstract}
 We examine two central regularization strategies for monotone variational inequalities, 
 the first a direct regularization of the operative monotone mapping, and the second via 
 regularization of the associated dual gap function.  A key link in the relationship between the 
 solution sets to these various regularized problems is the idea of exact regularization, which, 
 in turn, is fundamentally associated with the existence of Lagrange multipliers for 
 the regularized variational inequality.   A regularization is said to be exact if a solution to the 
 regularized problem is a solution to the unregularized problem for all parameters
 beyond a certain value.  The Lagrange multipliers corresponding to a particular 
 regularization of a variational inequality, on the other hand, are defined via the dual gap function.  
Our analysis suggests various conceptual, iteratively regularized
numerical schemes, for which we provide error bounds, and hence stopping criteria, under the
additional assumption that the solution set to the unregularized problem is what we call 
{\em weakly sharp} of order greater than one.   
\end{abstract}


\section{Introduction.}
Given a mapping   $F: \Rbb^n
\rightarrow \Rbb^n$, a closed set $\Omega\in \Rbb^n$, we consider the
variational inequality problem $\vi(F,\Omega)$:
\boxedtxt{
 find a vector $\xbar\in \Omega$ such that
\begin{equation*}\tag{VI}\label{e:VI}
 \langle F(\xbar), y-\xbar \rangle\geq 0,~~~\forall y\in \Omega
\end{equation*}
}
and the (strong) generalized variational inequality $\gvi(T,\Omega)$ 
for a multivalued mapping $T: \DRLRn \rightrightarrows \Rbb^n$:  
\boxedtxt{
find  $\xbar\in \Omega$ such that
\begin{equation*}\tag{GVI}\label{e:GVI}
~\exists~ v\in T(\xbar)\mbox{ with }
\DRLip{v}{x-\xbar}\geq 0\quad  \forall x\in \Omega.
\end{equation*}
}
We will denote the sets of solutions to these problems by $\sol(F,\Omega)$ and 
$\sol(T,\Omega)$ respectively where the corresponding problem, 
$\vi(F,\Omega)$ or \\
$\gvi(T,\Omega)$, is clear
from context.
% \begin{equation}\label{DRLe:GVI}
%  ~\forall x\in \Omega~\exists~ v(x)\in T(\xbar)\mbox{ with }
% \DRLip{v(x)}{x-\xbar}\geq 0
% \end{equation}
Though the subject of variational inequalities is well-established (see \cite{FacchineiPang} for
the basic theory and algorithms), we recall some basic definitions.
\begin{defn}\label{d:monotone}[(pseudo) monotone mappings]
 A mapping $F: \Rbb^n
\rightarrow \Rbb^n$ is said to be {\em pseudomonotone} on $\Omega$ if for all $x,y\in\Omega$ 
\begin{equation}\label{e:pseudomon}
   \langle x-y, F(y)\rangle\geq 0\quad \implies\quad    \langle x-y, F(x)\rangle \geq 0.
\end{equation}
$F$ is said to be 
pseudomonotone$^+$ on $\Omega$ if $F$ is pseudomonotone and for all $x, y$ in $\Omega$,
\begin{equation}
 \langle F(x), y-x \rangle\geq 0\quad\mbox{and}\quad \langle F(y), y-x \rangle=0 \Rightarrow F(y)=F(x).
 \end{equation}
$F$ is called {\em monotone} on $\Omega$ if
\begin{equation}
 \langle F(x)-F(y), x-y \rangle \geq 0~~~~\forall~x,y\in\Omega.
\end{equation}
$F$ is called {\em strongly monotone} on $\Omega$ if there exists a $\mu>0$ such that
\begin{equation}
 \langle F(x)-F(y), x-y \rangle \geq \mu \|x-y\|^2~~~~\forall~x,y\in\Omega.
\end{equation}
\end{defn}
We recall that for monotone functions, the solution set, if exists, is convex. Throughout this work we assume the following.
\begin{center}
\fbox{%
        \addtolength{\linewidth}{-2\fboxsep}%
        \addtolength{\linewidth}{-2\fboxrule}%
        \begin{minipage}{\linewidth}%
        \begin{assumption}\label{hyp:F}$~$
\begin{enumerate}[(i)]
   \item
\label{hyp:F1} 
$\Omega\subset\DRLRn$ is nonempty, closed and convex.
\item
\label{hyp:F2} $\DRLmap{F}{\Omega}{\DRLRn}$ is continuous and monotone.  
\end{enumerate}
\end{assumption}
 \end{minipage}%
}
\end{center}
Assumption \ref{hyp:F}\eqref{hyp:F2} above is understood in the context of Definition \ref{d:monotone} by the 
obvious extension of $F$ to a mapping defined on $\DRLRn$ by the mapping whose effective domain is $\Omega$, that is $F(x)=\emptyset$
for all $x\notin \Omega$.   
We recall a standard result on existence and boundedness of the set of solutions to $\vi(F,\Omega)$.  
Define the \textit{recession} or {\em asymptotic cone} $\Omega^\infty$ by
\begin{equation}\label{e:Omegainf}
\Omega^\infty\equiv \DRLset{w\in \DRLRn}{\mbox{ for any } x\in \Omega, ~x+w\tau \in \Omega~\forall ~\tau\geq 0}.
\end{equation}
\begin{lemma}[Exercise 12.52, \cite{VA}]\label{t:S_0 compact}
Under assumption \ref{hyp:F}, $\sol(F,\Omega)$ is nonempty and bounded if and only if 
\begin{equation}\label{hyp:F4}
w\in \Omega^\infty\setminus\{0\}~\implies
\exists x\in \Omega \mbox { with }\DRLip{F(x)}{w}>0.  
\end{equation}
\end{lemma}

There is a vast literature on how to solve a variational 
inequality under various assumptions (see \cite{FacchineiPang} and 
references therein).  Of particular interest for us are 
{\em ill-posed variational inequalities}.  There are many definitions of ill-posedness.  Here we will 
consider ill-posed any variational inequality $\vi(F,\Omega)$ for which $F$ is not strongly monotone.  \
The conventional approach to 
such problems is to {\em regularize}, or otherwise modify the problem so that the regularized  
problem is well-posed and has one or more solutions that are reasonable approximations to solutions to 
the original problem. A solution to the desired ill-posed problem, if exists, is then achieved as a limit 
of solutions to well-posed approximate problems.
  
A central motivation of this paper is the concept of {\em exact regularization} for a variational inequality, that is, 
a regularization for which the regularized solution corresponds to a solution to the unregularized problem
for all regularization parameters below a certain threshold.  Exact {\em penalization} is a well understood concept
in constrained optimization, and the relation to the existence of Lagrange multipliers has been extensively studied.  This has 
recently been extended to penalized variational inequalities where the connection to Lagrange multipliers also 
appears \cite[Lemma 4]{FacchineiPangLampariello}.   
We take our inspiration from the concept of exact {\em regularization} developed in 
the context of convex programming by Friedlander and Tseng \cite{FriedlanderTseng07}. Given a convex mapping   $f: \Rbb^n
\rightarrow \Rbb$, a nonempty closed convex set $\Omega\in \Rbb^n$ and 
a continuous convex map $\varphi:\Rbb^n
\rightarrow \Rbb$, $\epsilon>0$, consider the following regularization scheme:
\[
(\mathcal{P}_0)\qquad \underset{\substack{x \in \Omega}}{\mbox{minimize }}f(x)\quad\to \quad
(\mathcal{P}_\varepsilon)\qquad \underset{\substack{x \in \Omega}}{\mbox{minimize }}f(x)+\varepsilon\varphi(x).
\]
When $\varphi=\|.\|^2$, it is the well known Tikhonov regularization and when $\varphi=\|.\|_1$, an $l_1$ regularization.
The regularization is said to be exact if solutions to $(\mathcal{P}_\varepsilon)$ are solutions 
to $(\mathcal{P}_0)$ for $\varepsilon$ below some threshold value.  

Generalizing this to variational inequalities, for any continuous convex mapping 
$\DRLmap{\varphi}{\Rbb^n}{\Rbb\cup\{+\infty\}}$ with $\DRLdom \varphi=\Omega$, denote
$T_\epsilon\equiv F+\epsilon\DRLsd\varphi$, with $\epsilon>0$ fixed, and 
$F$ extended by $\emptyset$ to a mapping $\DRLRn\to \DRLRn$.   
We consider the following regularization strategy for $ \vi (F, \Omega) $ which turns
out to be a specialization of \eqref{e:GVI}:   
\boxedtxt{
find $\xbar \in \Omega$ such that
\begin{equation*}\tag{GVI$_{T_\epsilon}$}\label{e:GVIe}
 \exists~ v\in T_\epsilon(\xbar)\equiv F(\xbar) + \epsilon\DRLsd\varphi(\xbar)\mbox{ with }
\DRLip{v}{x-\xbar}\geq 0, ~\forall x\in \Omega.
\end{equation*}
}
In our extension of the notion of exact regularization to variational inequalities we 
introduce {\em Lagrange multipliers} for variational inequalities, the existence
of which are closely tied to the existence of exact regularization strategies.  
The central tool for our analysis is the {\em gap function}.

For a given variational
inequality $\vi(F,\Omega)$, a {\em gap function} is a function $\psi:\Rbb^n \rightarrow \Rbb\cup\{+\infty\}$
with $\Omega\subseteq \DRLdom \psi$ and 
\begin{enumerate}
 \item $\psi(x)\geq 0$ for all $x \in \Omega$;
 \item $\psi(\xbar)=0$, $\xbar \in \Omega$ if and only if $\xbar$ solves $\vi(F,\Omega)$.
\end{enumerate}
It is clear that any minimizer $\bar{x}$ of the gap function $\psi$ over $\Omega$
with $\psi(\bar{x})=0$ is a solution to $\vi(F,\Omega)$.
The first occurrence of the gap function for $\vi(F,\Omega)$ is  
Auslender's gap function \cite{Aus73}:
\begin{equation}\label{e:gap func}
 \theta(x)=\sup_{y\in \Omega} \langle F(x), x-y \rangle,
\end{equation}
where we use the convention that the value of a function on the emptyset is $+\infty$ so that $\theta(x)=+\infty$
at points $x\notin\Omega$. 
The dual gap function for $\vi(F,\Omega)$ is given by
\begin{equation}\label{e:dual gap func}
 G(x)=\sup_{y\in \Omega} \langle F(y), x-y \rangle.
\end{equation}
For each fixed $y$, the function $x\mapsto \langle F(y), x-y \rangle$ is affine. Thus the dual gap function $G$ is 
closed and convex on $\Omega$ since it is the pointwise supremum over affine functions.
The dual gap function is not necessarily a gap function for $\vi(F,\Omega)$, however, with additional assumptions
on $F$, it is indeed a gap function.  In particular,  
if the mapping $F$ is pseudomonotone and continuous, then $G$ is 
in fact a gap function for $\vi(F,\Omega)$ \cite[Theorem 2.3.5]{FacchineiPang}. 

Note that neither $\theta$ nor $G$ is finite valued in general. If $\Omega$ is assumed to be
compact, then both are finite-valued, but we will avoid such restrictions in what follows. 
A regularized gap function for $\vi(F,\Omega)$ with regularization parameter $\alpha>0$, 
is given by 
 \begin{equation}\label{e:reggap}
  \theta_\alpha(x)=\sup_{y\in \Omega} \left\{\langle F(x), x-y \rangle
-\frac{\alpha}{2}\|y-x\|^2\right\}.
 \end{equation}
This was introduced in \cite{Fukushima92} and is finite valued for any closed
convex set $\Omega$. Note that
when $F$ is strongly monotone then $G$ is finite valued even without $\Omega$ being compact. For, the
strong monotonicity of $F$ on $\Omega$ with constant $\mu$ implies that
\begin{equation}
 \langle F(y), x-y\rangle \leq \langle F(x), x-y\rangle -\mu\|x-y\|^2.
\end{equation}
Hence
\begin{equation}
 G(x)=\sup_{y\in \Omega }\langle F(y), x-y\rangle \leq  \sup_{y\in \Omega }\{ \langle F(x), x-y\rangle -\mu\|x-y\|^2\}=\theta_{2\mu}(x)<\infty.
\end{equation}

One can reformulate $\vi(F,\Omega)$ as a constrained optimization problem using
$\theta_\alpha$. 
Since the objective function in \eqref{e:reggap} is strongly concave, for every
$x$ there exists a unique solution $y_\alpha(x)$ which is explicitly given by
\begin{equation}
 y_\alpha(x)=P_\Omega\left(x-\frac{1}{\alpha}F(x)\right)
\end{equation}
where $P_\Omega(z)\equiv\argmin_{y\in\Omega}\|y-z\|$ is the {\em projection} onto the 
set $\Omega$. 
Hence $\theta_\alpha(x)$ can be explicitly written as
 \begin{equation}
  \theta_\alpha(x)=\left\{\langle F(x), x-y_\alpha(x) \rangle
-\frac{\alpha}{2}\|y_\alpha(x)-x\|^2\right\}.
 \end{equation}
When $F$ is continuously differentiable, $\theta_\alpha$ is continuously differentiable
\cite[Theorem 3.2]{Fukushima92} and hence we can reformulate
$\vi(F,\Omega)$ as a constrained 
optimization problem with the differentiable objective function
$\theta_\alpha$.

We show in Section \ref{s:existence and boundedness} that, although the solution set of the regularized 
problem \eqref{e:GVIe} has some relation to the solution set of \eqref{e:VI}, we will achieve a more precise 
correspondence via the dual gap function 
$G$ defined by \eqref{e:dual gap func} and the equivalence between solutions to the problem 
\eqref{e:VI} and the convex optimization 
problem 
\begin{equation}\tag{$\mathcal{P}_G$}\label{e:PG}
\underset{\substack{x \in \Omega}}{\mbox{minimize }}~ G(x).
\end{equation}
If $\sol(F,\Omega)\neq \emptyset$, then solving $ \vi(F, \Omega) $ is 
equivalent to solving \eqref{e:PG}.
The corresponding regularization of the the above convex optimization problem in the spirit of \cite{FriedlanderTseng07}
gives us the problem 
\begin{equation}\tag{$\mathcal{P}_{G_{\varepsilon\varphi}}$}\label{e:PGe}
\underset{\substack{x \in \Omega}}{\mbox{minimize }}G_{\varepsilon\varphi}(x)\equiv G(x) + \varepsilon \varphi(x).
\end{equation}

\begin{defn}[exact regularization of variational inequalities]\label{d:exact reg VI}
A regularization of the variational inequality \eqref{e:VI} is said to be {\em exact} 
if solutions to the convex optimization 
problem \eqref{e:PGe} are
also solutions to  \eqref{e:VI} for all values of $\epsilon$
below some threshold value $\epsilonbar>0$.   
\end{defn}

Another advantage of gap functions is the availability of computable error bounds for 
strongly monotone variational inequalities. Error bounds, in turn, are essential for principled 
stopping criteria for algorithms.  These are discussed in Section \ref{s:reg G} where 
we derive an upper bound on the error under the assumption that the solution set $\vi(F,\Omega)$ is 
{\em weakly-sharp of order gamma} \eqref{weaksharpVIG-gamma}.
Error bounds can also be achieved for the special 
case of monotone mappings where $F(x)=Mx+q$ and $\Omega=\Rbb^n_+$ or a polyhedron with a positive 
semidefinite matrix $M$.  For general monotone variational inequalities, however, we are unaware of any 
results on error bounds using the gap function. 

Through the study of unconstrained reformulations for variational inequalities  
the closely related $D$-gap function $\theta_{\alpha\beta}$ for $\vi(F,\Omega)$ was introduced \cite{Peng97}.
It is defined as the difference of two regularized
gap functions $\theta_{\alpha}$ and $\theta_{\beta}$ with $\beta>\alpha$ and  is given by
\begin{equation}\label{e:Dgap}
 \theta_{\alpha\beta}(x)= \theta_{\alpha}(x)-\theta_{\beta}(x); ~~\beta>\alpha>0.
\end{equation}
The D-gap function satisfies the following properties \cite[Theorem 3.2]{YamashitaTajiFukushima97}:
\begin{enumerate}
\item $ \theta_{\alpha\beta}(x)\geq 0$ for all $x \in \Rbb^n$;
 \item $ \theta_{\alpha\beta}(\xbar)=0$, $x\in \Rbb^n$ if and only if $\xbar$ solves $\vi(F,\Omega)$.
\end{enumerate}
The D-gap function provides an unconstrained
reformulation of the variational 
inequality \cite[Theorem 3.2]{YamashitaTajiFukushima97}.    
As with the gap function, when $F$ is continuously differentiable, the D-gap function is smooth 
and the resulting unconstrained optimization 
problem of minimizing $\theta_{\alpha\beta}$ is smooth \cite[Theorem 3.1]{YamashitaTajiFukushima97}. 

As the theory for gap and $D$-gap functions for generalized variational inequalities is underdeveloped, particularly 
with regard to numerical algorithms, we will, when necessary, restrict our attention to {\em differentiable}
strongly convex regularizers $\varphi$.  Our numerical approach for solving the regularized problems 
$\vi(T_\epsilon,\Omega)$ with $T_\epsilon=F+\epsilon\nabla\varphi$ is via D-gap functions for  
which there is ample choice of appropriate methods.  We use the attendant error bounds developed 
in \cite{Dutta12} and \cite{YamashitaTajiFukushima97} for iterative methods for solving 
$\vi(T_\epsilon, \Omega)$ with $\epsilon$ fixed and $T_\epsilon=F+\epsilon\nabla\varphi$ strongly monotone. 
In the limit as $\epsilon\to 0$ we approach the solution set to $\vi(F,\Omega)$.  If our regularization $\varphi$
is {\em exact}, then, for some $\epsilon$ below a threshold value, the procedure for solving the 
regularized problem will converge to a point in $\sol(F,\Omega)$  with computable error bounds. 
When the regularization is not exact, error bounds on the distance from the regularized solution to the
original solution set is provided in \cite{FriedlanderTseng07} for a general convex minimization. 
We derive in Section \ref{s:error bounds S_eps}  conditions for a 
similar error bound for the generalization to variational inequalities. 
Unlike the case of convex minimization our generalization demands more than the existence of
the "weak sharp minima"  in order to achieve exact regularization.  Our characterization 
\eqref{weaksharp-gamma-stronger} appears to be new. 

In section 2, we study the properties of the solution sets of generalized variational inequalities  with the purpose of
understanding the solution sets of the regularized problem $ \gvi(T_\epsilon, \Omega) $ with 
$T_\epsilon=F+\epsilon\partial\varphi$.  Here the essential role of the dual gap function for 
characterizing exact regularization becomes apparent. In Section 3 we focus on  
 the solution methods for monotone variational inequalities via iterative regularization 
of the dual gap function. The analysis in Section 4 is refined to the special case when $\varphi$ is differentiable
, where we study direct regularization of the variational inequality via \eqref{e:GVIe}.
In the same section we present some numerical results illustrating the theory.
 
\section{Solution Sets}\label{s:existence and boundedness}
We begin with a study of the relationship between regularized generalized variational inequalities and 
their limit as the regularization parameter goes to zero.  
\subsection{Basic Facts, Notation and Assumptions}
Our focus in this section is on the solution sets of  $\gvi(T, \Omega) $ where $ T $ is a {\em maximal 
monotone map} and $ \Omega $ is a non-empty closed and convex set. 
\begin{defn}[normal cone]
A Normal cone to a closed convex set $\Omega$ at a point $\xbar \in \Omega$ is defined as
\begin{equation} 
 N_{\Omega}(\xbar)=\{v\in \mathbb{R}^n: \langle v,  x-\xbar \rangle \leq 0 \quad \forall x\in \Omega\}.
\end{equation}

\end{defn}

\begin{defn}[(maximal) monotone mappings]\label{d:maximonotone}
A set-valued map $\DRLmmap{T}{\Rbb^n}{\Rbb^n}$ is {\em $\xi$-monotone} for some
$\xi>1$ if there exists
$\mu>0$ such that
\begin{equation}
 \langle v-w, x-y \rangle \geq \mu \|x-y\|^\xi~~~~\forall~(x, v) \in \gph T~
\mbox{ and }\forall ~ (y,w)\in \gph T.
\end{equation}
It is simply said to be {\em monotone} if it is $\xi=1$ and $\mu=0$ in the above equation.
$F$ is {\em maximally monotone} if there is no monotone operator 
$\DRLmmap{\Tbar}{\DRLRn}{\DRLRn}$ such that the graph of $\Tbar$ 
properly contains the graph of $T$. 
$T$ is {\em strongly monotone} if there
exists
$\mu>0$ such that
\begin{equation}
 \langle v-w, x-y \rangle \geq \mu \|x-y\|^2~~~\forall~(x, v) \in \gph T~
\mbox{ and }\forall~  (y,w)\in \gph T.
\end{equation}
\end{defn}
Using this notion we can alternatively write \eqref{e:GVI}
as a maximal monotone inclusion: 
\begin{eqnarray*}
0 \in T(x) + N_\Omega (x).
\end{eqnarray*}
\noindent Note that $N_\Omega$ is maximal monotone \cite[Example 20. 41]{BauschkeCombettes}. If $\DRLdom T = \Omega $ then $ T + N_\Omega $ is also maximal monotone. More generally, if  
$ \rint (\DRLdom T)  \cap \rint (\DRLdom  N_\Omega) = \rint (\DRLdom T) \cap \rint (\Omega) \neq \emptyset $ then $ T + N_\Omega $
is also maximal monotone. 

Another central property of set-valued mappings that we will make use of concerns the notion 
of continuity.  
\begin{defn}\label{d:osc}
A map $T: \Rbb^n\rightrightarrows \Rbb^n$ is {\em outer semi-continuous} at
a point $\xbar\in \Rbb^n$ if 
\[
\bigcup_{x^k\to \xbar}\limsup_{k\to\infty} T(x^k)\subset T(\xbar) 
\]
\end{defn}
In another useful characterization, a set-valued map $T$ is  outer semicontinuous everywhere if and only if its graph is closed
(\cite[Theorem 5.7]{VA}).\\\\

We begin by studying some of the fundamental properties of the solution set $\sol(T,\Omega)$, 
such as convexity and boundedness.  We begin with convexity.  For  this, 
we introduce the notion of a {\em Minty $GVI$}, denoted by  $\mgvi(T, \Omega)$,  wherein 
we seek $\xbar\in \Omega$, such that for each $y\in \Omega$ 
and any $v\in T(y)$
\begin{equation}
\langle v, y-\xbar \rangle \geq 0.
\end{equation}
Compare this to the {\em weak $GVI$}, denoted by
$\wgvi(T, \Omega)$ \cite{AusselDutta11}, wherein we seek to find $\xbar$, such that for 
each $x\in \Omega$, there exists $w_x\in T(\xbar)$ such that
\begin{equation}
 \langle w_x, x-\xbar \rangle\geq 0.
\end{equation}
Let us denote the solution sets of $\mgvi(T, \Omega)$ and $\wgvi(T, \Omega)$ by
$\sol^M(T, \Omega)$ and $\sol^W(T, \Omega)$. We will first show that $\sol^M (T, \Omega)$ is a convex set.
\begin{lemma}\label{l:minty-weak}
 $\sol^M (T, \Omega)$ is a convex set.
 \begin{enumerate}[(i)]
  \item\label{l:minty-weak i} If T is monotone, then $\sol (T, \Omega)\subseteq \sol^M (T, \Omega)$.
    \item\label{l:minty-weak ii} If T is locally bounded and graph closed, then $\sol^M (T, \Omega) \subseteq \sol^W (T, \Omega)$. 
 \end{enumerate} 
\end{lemma}
\proof
Let $\xbar_1, \xbar_2 \in \sol^M (T, \Omega)$ and let $y\in \Omega$. Then for any $v\in T(y)$,
\begin{eqnarray}
\label{eq1}\langle  v, y-\xbar_1 \rangle &\geq & 0\\
\label{eq2}\langle  v, y-\xbar_2 \rangle &\geq & 0.
\end{eqnarray}
Now multiplying \eqref{eq1} with $\lambda$ and \eqref{eq2} with $(1-\lambda)$ where $0\leq\lambda\leq 1$, we have 
\begin{equation}
 \langle v, y-(\lambda \xbar_1+(1-\lambda)\xbar_2) \rangle\geq 0.
\end{equation}
Since $y\in \Omega$ was chosen arbitrarily, we have
\begin{equation}
 \lambda \xbar_1+(1-\lambda)\xbar_2\in \sol^M (T, \Omega).
\end{equation}
Hence $\sol^M (T, \Omega)$ forms a convex set.\\\\ 
Part \eqref{l:minty-weak i}. By invoking the monotonicity of $T$, it is simple to show that
\begin{equation}
 \sol (T, \Omega)\subseteq \sol^M (T, \Omega).
 \end{equation}
Part \eqref{l:minty-weak ii}. Let $\xbar \in \sol^M(T,\Omega))$. Then for any $y\in \Omega$ and $v\in T(y)$
\begin{equation}
 \langle v, y-\xbar \rangle \geq 0.
\end{equation}
Let us construct the sequence \[y_n=\xbar+\frac{1}{n}(x-\xbar),\]
where $x\in \Omega$ is a fixed but arbitrary point. Of course we have
$y_n\in \Omega$, since $\Omega$ is a closed convex set. Hence for any $v_n\in T(y_n)$
we have
\begin{equation}
 \langle v_n, y_n-\xbar \rangle \geq 0, \quad \forall n\in \Nbb
\end{equation}and hence \begin{equation}
 \langle v_n, x-\xbar \rangle \geq 0, \quad \forall n\in \Nbb.
\end{equation}
As $T$ is locally bounded, by noting that $y_n\rightarrow \xbar$, we can conclude that $v_n$ is 
a bounded sequence. Without loss of generality let us assume that $v_n\rightarrow v_x$. Note that the limit must depend on
the chosen $x$. Hence we have 
\begin{equation}
 \langle v_x, x-\xbar \rangle \geq 0.
\end{equation}
Further, as $T$ is graph closed, $v_x\in T(\xbar)$. Now, this limit $v_x$ will change with $x$. Since $x\in \Omega$ is arbitrary, the above argument
can be repeated for each $x\in \Omega$. This shows that $\xbar\in \sol^W(T, \Omega)$.\hfill$\Box$

The next result determines appropriate conditions on $T$ that guarantee that
\[ \sol^M (T, \Omega)\subseteq\sol (T, \Omega),\]
from which it follows by Lemma \ref{l:minty-weak} that $\sol (T, \Omega)$ is a convex set.
\begin{thm}
Let $T:\Rbb^n\rightrightarrows\Rbb^n$ be a non-empty, convex and compact valued map. Further, assume that $T$ is monotone, locally 
bounded and graph closed. Then $sol(GVI(T,\Omega))$ is a convex set.
\end{thm}
\proof Using the fact that $T$ is compact-valued, Aussel and Dutta \cite{AusselDutta11} had constructed the following gap function for 
$WGVI(T, \Omega)$. This given as
\begin{equation*}
 \ghat(x)=\sup_{y\in \Omega}\inf_{v\in T(x)}\langle  v, x-y\rangle.
\end{equation*}
Let $\xbar\in \sol^M(T, \Omega)$. Note that any $\xbar\in \sol^W(T, \Omega)$ satisfies $\ghat(\xbar)=0$. Since $\sol^M(T, \Omega)\subseteq \sol^W(T, \Omega)$ from Lemma \ref{l:minty-weak}, we can now write
\begin{equation*}
 \ghat(\xbar)=0=\sup_{y\in \Omega}\inf_{v\in T(\xbar)}\langle v, \xbar-y \rangle.
\end{equation*}
Now since $\xbar$ is fixed, the function $v\mapsto \langle v, \xbar-y\rangle$ is linear for each fixed $y$
 and the function $y\mapsto \langle v, \xbar-y\rangle$ is affine (and hence concave) for each fixed $v$.
Hence, as $T(\xbar)$ is convex and compact valued we can invoke the famous Sion's minimax theorem to conclude that 
\begin{equation*}
 0=\inf_{v\in T(\xbar)}\sup_{y\in\Omega}\langle v, \xbar-y\rangle 
\end{equation*}
Let \begin{equation*}
     \varsigma(v,\xbar)=\sup_{y\in \Omega}\langle v, \xbar-y \rangle
    \end{equation*}
    Note that for each $y\in \Omega$, as $\langle v, \xbar-y \rangle$ is linear we conclude that $\varsigma(v, \xbar)$ 
    is in convex in $v$ and lower semicontinous. Moreover, $\varsigma(v, \xbar)$ is a proper function since 
    \[0=\inf_{v\in T(\xbar)}\varsigma(v,\xbar)\]
    Thus as $T(\xbar)$ is convex and compact we conclude the existence of $v^*\in T(\xbar)$ such that
    \begin{equation*}
     0=\varsigma(v^*, \xbar)
    \end{equation*}
Hence \begin{equation*}
       \sup_{y\in \Omega}\langle v^*, \xbar-y\rangle=0
      \end{equation*}
Thus for all $y\in \Omega$ we have
\begin{equation*}
 \langle v^*, \xbar-y \rangle\leq 0
\end{equation*}
or
\begin{equation*}
 \langle v^*, y-\xbar \rangle\geq 0
\end{equation*}
This shows that $\xbar\in \sol(T,\Omega)$ and hence $\sol^M(T,\Omega)\subseteq \sol(T,\Omega)$.
Using Lemma \ref{l:minty-weak} we conclude that 
\begin{equation*}
 \sol^M(T,\Omega)= \sol^W(T,\Omega)
\end{equation*}
Therefore, again from Lemma \ref{l:minty-weak} $\sol(T,\Omega)$ is a convex set.\hfill$\Box$\\

Next, we determine conditions that guarantee boundedness of $\sol(T,\Omega)$. 
\begin{propn}[existence and boundedness of $\sol(T,\Omega)$]
\label{t:boundedness of solGVI}
Let $\Omega\subset\DRLRn$ be closed convex and nonempty, let $\DRLmmap{T}{\DRLRn}{\DRLRn}$
be maximal monotone  with $\DRLdom T=\Omega$.  
The set of solutions to $\gvi(T,\Omega)$  is  nonempty and bounded
if and only if 
\begin{equation}\label{hyp:existence and bounded}
   w\in \Omega^\infty\setminus\{0\}~\implies ~\exists x\in \Omega\mbox{ with }
\DRLip{v}{w}>0~\mbox{ for some }v\in T(x).
\end{equation}
If $\Omega$ is bounded, then $\Omega^\infty=\{0\}$ and the implication holds trivially. 
\end{propn}
\begin{proof}
   The proof follows from 
\cite[Theorem 12.51]{VA} in a minor extension of \cite[Exercise 12.52]{VA}.  We show that
\eqref{hyp:existence and bounded} is equivalent to the existence of 
$v\in \DRLrange (T + N_\Omega) $ with  $\DRLip{v}{w}>0$ for each nonzero $w\in \Omega^\infty$.  
Existence and boundedness of the solution set to $\gvi(T,\Omega)$ then follows directly from 
\cite[Theorem 12.51]{VA}, since the solution set of $ \gvi (T, \Omega) $ coincides with the set
$ ( T + N_\Omega)^{-1} (0) $.\\

Indeed,  if $\Omega^\infty\setminus\{0\}$ is empty then $\Omega$ is bounded and there is nothing to prove.  
Suppose, then, that $w\in \Omega^\infty\setminus\{0\}$.  For each $x\in \rint \Omega$ and for all $\tau>0$ we can write 
$w=\frac{(x_\tau-x)}{\tau}$ for  some $x_\tau\in \Omega$, hence $w\in T_{\Omega}(x)$, the {\em tangent cone} to 
$\Omega$ for all $x\in \Omega$ \cite[Definition 6.25 and Corollary 6.29]{VA}.  Hence 
\begin{equation}\label{e:tangent cone}
(\forall x\in \Omega)\quad   \DRLip{w}{z}\leq 0\quad\mbox{ for all }z\in N_{\Omega}(x).
\end{equation}
Now, by \cite[Theorem 12.51]{VA} $(T+ N_\Omega)^{-1} (0)$ -- that is the solution set to $\gvi(T,\Omega)$ -- 
is nonempty and bounded 
if and only if for each nonzero $w\in (\DRLdom ( T + N_\Omega))^\infty=\Omega^\infty$ there exists 
$ \hat{v} \in \DRLrange (T + N_\Omega)$ 
with  $\DRLip{\hat{v}}{w}>0$.   This means that there exists $ x \in \DRLdom T \cap \Omega $, 
$ v \in T(x) $  and $ z \in N_\Omega(x) $ such that
$ \hat{v} = v + z $ and 
\begin{eqnarray*}
\langle v + z , w \rangle > 0 .
\end{eqnarray*}
Since $ \langle z, w \rangle \le 0 $ it follows that $ \langle v, w \rangle > 0 $. 
This is exactly the statement in \eqref{hyp:existence and bounded}. 
\end{proof}

To guarantee maximal monotonicity 
of the related set-valued mapping, which is central to the application of 
Proposition \ref{t:boundedness of solGVI},    
we will restrict our attention to regularizing functions $\varphi$ satisfying the following 
assumption.
\begin{center}
\fbox{%
        \addtolength{\linewidth}{-2\fboxsep}%
        \addtolength{\linewidth}{-2\fboxrule}%
        \begin{minipage}{\linewidth}%
\begin{assumption}\label{hyp:GVIepsilon}
$~$
\begin{enumerate}[(i)]
\item
\label{hyp:varphi} $\DRLmap{\varphi}{\Omega}{\Rbb}$ is continuous and convex.  
\item
\label{hyp:rint epsilon} 
$   0\in \rint(\DRLdom \DRLsd\varphi - \Omega)$.
\end{enumerate}
\end{assumption}
 \end{minipage}%
}
\end{center}

An understanding of convergence of solutions to 
\eqref{e:GVIe} to the unregularized monotone problem \eqref{e:VI} is achieved through 
the solution set to the following generalized variational inequality.
\boxedtxt{
Find $\xbar\in \sol(F,\Omega)$ such that  
\begin{equation*}\tag{GVI$\varphi$}\label{e:GVIphi}
\exists~ v\in \DRLsd\varphi(\xbar)\mbox{ with } \DRLip{v}{x-\xbar}\geq 0 \quad\forall~x\in \sol(F,\Omega).
\end{equation*}
}
To achieve compactness of the problem \eqref{e:GVIphi} we will require the following assumption. 
\begin{center}
\fbox{%
        \addtolength{\linewidth}{-2\fboxsep}%
        \addtolength{\linewidth}{-2\fboxrule}%
        \begin{minipage}{\linewidth}%
\begin{assumption}\label{hyp:GVIvarphi}
$~$
\begin{enumerate}[(i)]
   \item
\label{hyp:1} 
$\sol(F,\Omega)$ is nonempty and closed.
\item
\label{hyp:rint_varphi} 
$   0\in \rint(\DRLdom \DRLsd\varphi - \sol(F,\Omega))$.
% \Omegaor. 24.4 BauschkeCombettes
\end{enumerate}
\end{assumption}
 \end{minipage}%
}
\end{center} 
\begin{defn}
 The indicator function $\iota_C$ of a set $C$ is defined by
 \begin{equation*}
  \iota_C(x)=
   \begin{cases}
 0,& if~x \in C,\\ +\infty ,& x \not\in C.
 \end{cases}
 \end{equation*}
\end{defn}
Note that for a closed convex set $C$, the subdifferential of the indicator function is the normal cone
\begin{equation*}
N_C(x)= \partial (\iota_C)(x)=\begin{cases}
 \{v \in \mathbb{R}^n: \langle v, y-x \rangle \leq 0 \quad \forall y\in C\},&if~x \in C,\\ \emptyset,& x \not\in C.
 \end{cases}
\end{equation*}
\begin{cor}\label{t:Sphi compact}
Let $S_0$ denote the solution set to $\vi(F,\Omega)$. 
Under Assumptions \ref{hyp:F}, \ref{hyp:GVIepsilon} and \ref{hyp:GVIvarphi}, 
the solution set $\sol(\partial\varphi, S_0)$ 
is nonempty, bounded and convex if and only if 
for each $w\in S_0^\infty\setminus\{0\}$, if any, there is an $x\in S_0$ with 
\begin{equation}\label{hyp:4}
\DRLip{v}{w}>0 \mbox{ for some } v\in \DRLsd\varphi(x).  
\end{equation}
\end{cor}
\begin{proof}
    By Assumption \ref{hyp:GVIvarphi}\eqref{hyp:1} 
the solution set $S_0$ is closed and nonempty.  Furthermore, 
$S_0$ is convex for $\Omega$ convex by the monotonicity 
and continuity of $F$ (Assumption \ref{hyp:F}\eqref{hyp:F2}).  The solution set $\sol(\partial \varphi, S_0)$ can 
be characterized as $T^{-1}(0)\equiv\DRLset{x\in\DRLRn}{0\in T(x)}$ where 
$T\equiv\DRLsd\varphi+ N_{S_0}$.  
For $S_0$ closed convex, the 
normal cone mapping $N_{S_0}=\DRLsd\iota_{S_0}$ is  maximal monotone,  
and for $\varphi$ continuous and convex on $\Omega$ with 
\[
0\in \rint \left( \DRLdom \DRLsd\varphi - S_0\right)   
\subset\rint \left( \DRLdom \DRLsd\varphi - \Omega\right)
\]
(Assumptions \ref{hyp:GVIepsilon}\eqref{hyp:varphi} and \ref{hyp:GVIvarphi}\eqref{hyp:rint_varphi})
the operator $T$ is maximal monotone \cite[Corollary 24.4]{BauschkeCombettes}.   
That $\sol(\partial\varphi, S_0)$ is nonempty and bounded then follows from 
Proposition \ref{t:boundedness of solGVI}.  
To see that $\sol(\partial\varphi, S_0)$ is convex, 
note that  $\sol(\partial\varphi, S_0) = \argmin_{S_0}\varphi$, the solution set to a convex optimization 
problem, and thus   $\sol(\partial\varphi, S_0)$ is convex.  This completes the proof. 
\end{proof}

Condition \eqref{hyp:existence and bounded}  holds in particular for {\em coercive} mappings with 
respect to $\Omega$.
\begin{defn}[coercive mappings]\label{d:coercive} 
A mapping $\DRLmmap{T}{\Rbb^n}{\Rbb^n}$ is said to be {\em coercive with 
respect to $\Omega$} if, for any $x_0\in \Omega$ and for some $\gamma>0$,  
\begin{equation}\label{e:coercive T}
   \liminf_{\|x\|\to\infty}\frac{\DRLip{v}{x-x_0}}{\|x\|^\gamma}>0\quad \forall~v\in T(x).
\end{equation}
\end{defn}
The above definition uses the convention that the infimum over an empty set is $+\infty$ in the case that 
$T(x)=\emptyset$ (and in particular, if $\Omega$ is bounded and $\DRLdom T=\Omega$). 

\begin{propn}[existence and boundedness with coercivity]\label{t:coercive}
If ~$T$ satisfies \eqref{e:coercive T}, 
then \eqref{hyp:existence and bounded} holds.  
Moreover, if $T$ is
maximally monotone with $\DRLdom T=\Omega$ then
 \eqref{e:coercive T} is sufficient for   
$\sol(T,\Omega)$ to be nonempty and bounded.
\end{propn}
\begin{proof}
  Let us set $ \hat{T} = T + N_\Omega $ .  If $\Omega^\infty\setminus\{0\}$ is empty, then there is nothing to prove.  
So let $w\in \Omega^\infty\setminus\{0\}$ and define $x=x_0+w\tau\in \Omega$ for $x_0\in \rint\Omega$ and $\tau>0$.     
The inequality \eqref{e:coercive T} is equivalent to the existence of a constant $c>0$ such that for all 
$x$ large enough 
\begin{equation}\label{e:coercive T2}
   \DRLip{v}{x-x_0}\geq c\|x\|^\gamma\quad\forall~v\in  \hat{T}(x). 
\end{equation}
But this 
is equivalent  to 
\begin{equation}\label{e:coercive T3}
   \DRLip{v}{w}\geq \frac{c\|x_0+w\tau\|^\gamma}{\tau}>0\quad\forall~v\in \hat{T}(x). 
\end{equation}
The rest follows from Proposition \ref{t:boundedness of solGVI}.
\end{proof}\\

\subsection{Lagrange Multipliers for Variational Inequalities}
Our main result shows the relationship between exact regularization and Lagrange multipliers.  
What is meant by the latter is developed next.
To reduce clutter we will use the notation 
\boxedtxt{
\begin{eqnarray*}
&~&  S_0\equiv\sol(F,\Omega),\quad S_\epsilon\equiv\sol(T_\epsilon, \Omega)\mbox{ for } T_\epsilon\equiv F+\epsilon\DRLsd\varphi\\
&&S_\varphi\equiv\sol(\DRLsd\varphi,S_0)\und \SGeps\equiv\argmin_{\Omega}\{G+\varepsilon\varphi\}.
\end{eqnarray*}
\nonumber}

As noted in the proof of Corollary \ref{t:Sphi compact}, $S_\varphi$ is convex for $\varphi$ convex and $S_0$ convex.  
Moreover, for $ G $ defined by \eqref{e:dual gap func} (the dual gap function associated with $ \vi (F, \Omega) $), we 
have
\[
S_0=\{x\in\Omega~|~G(x)=0\}.
\]
Thus $\argmin_{S_0}\varphi$ is equivalent to the solution of the following convex programming problem
\begin{equation}\tag{$\mathcal{P}_{\varphi,G}$}\label{e:PphiG}
\underset{\substack{x \in \Omega}}{\mbox{minimize }} \varphi(x)\quad\mbox{subject to}\quad G(x) \leq 0.
\end{equation}
\noindent Problem \eqref{e:PphiG} is then a convex program whose solution set coincides with 
$S_\varphi$. 
as can be seen by the equivalence of $ GVI ( \partial \varphi, S_0 ) $  to the problem of 
minimizing the convex function $ \varphi $ over the convex set $ S_0 $.

The Lagrangian associated with \eqref{e:PphiG} is 
\begin{equation}\label{e:Lagrangian}
   L(x,\lambda)\equiv \varphi(x) + \lambda G(x)-\sigma_{\Rbb_-}(\lambda)+\iota_\Omega(x)
\end{equation}
where $\iota_\Omega$ is the indicator function of $\Omega$ and 
$\sigma_{\Rbb_-}$ is the {\em support function} -- equivalently, the Fenchel conjugate of the 
indicator function-- of the negative orthant.
The optimality condition for \eqref{e:PphiG} in Lagrangian form is then 
(see, for example \cite[Chapter 11, Section I]{VA})
\begin{equation}\label{e:Lagrange optimality}
0 \in \partial \varphi(\xbar) + \lambdabar \partial G(\xbar) + N_\Omega(\xbar)\quad\mbox{ for some }
\lambdabar\in N_{\Rbb_-}(G(\xbar)).
\end{equation}
Implicitly, we are assuming that $\partial G(\xbar)\neq \emptyset$.  
This leads naturally to the following definition.
\begin{defn}[Lagrange multiplier for variational inequalities]\label{d:Lagrange VI}
Let $S_0$ be the solution set to $\vi(F,\Omega)$ and 
$\DRLmap{\varphi}{\DRLRn}{\Rbb\cup\{+\infty\}}$.  Let $G$ be the dual gap function associated
with $\vi(F, \Omega)$ defined by \eqref{e:dual gap func}.
A Lagrange multiplier of the generalized variational inequality $\gvi(\partial\varphi, S_0) $ is a constant 
$ \lambda \ge 0 $ that is also a Lagrange multiplier of the convex programming problem \eqref{e:PphiG}, when 
it exists.
\end{defn}
Regarding existence, if 
$\argmax _{y \in \Omega}\langle F(y), \xbar - y \rangle \neq \emptyset $ then $ F(\ybar) \in \partial G(\xbar) $ where
$ \ybar \in  \mbox{argmax} _{y \in \Omega} \langle F(y), \xbar - y \rangle$. 
The argmax always exists if, for instance, $ \Omega $ is compact.  We will attain existence, instead, under 
less restrictive conditions. 
\begin{propn}\label{t:existence}
 Suppose $\Omega$ and $F$ satisfy Assumption \ref{hyp:F} and let $F$ be coercive on $\Omega$.
Then \begin{enumerate}[(i)]
        \item\label{t:existence i} $S_0$ is nonempty and bounded and
\item\label{t:existence ii}  the dual gap function 
$
G(x)=\sup_{y\in \Omega}\langle F(y), x-y\rangle   
$ 
is finite valued
for all $x\in\Omega$.
Consequently, the supremum is attained and $\partial G(x)\neq \emptyset$ for all $x\in\Omega$. 
\end{enumerate}
\end{propn}
\proof
\eqref{t:existence i}. Since $F$ is continuous and monotone with $\DRLdom F=\Omega$ it is, 
in fact, maximally monotone.  The statement then follows from Proposition \ref{t:coercive}. 

\eqref{t:existence ii} For the second statement, let us assume, on contrary, that $G(\xbar)=+\infty$ for some $\xbar\in \Omega$. Then, 
there exists a sequence $y^k\in \Omega$ such that $\lim_{k\rightarrow\infty}\langle F(y^k), \xbar-y^k\rangle=\infty$.
Since $\Omega$ is closed and $F$ is continuous on $\Omega$, it must be that $\|y^k\|\to\infty$ as $k\to\infty$.  
Thus, for $R_k\equiv\|y^k\|$,
\[
   -\infty =   - \lim_{k\rightarrow\infty}\langle F(y^k), \xbar-y^k\rangle
\geq \liminf_{y\in\Rbb^n\setminus\Bbb_{R_k},~ R_k\rightarrow\infty}\langle F(y), y-\xbar\rangle.
\]
Hence, for any fixed $\gamma>0$,
\begin{eqnarray*}
\liminf_{y\in\Rbb^n\setminus\Bbb_{R_k},~ R_k\rightarrow\infty}\frac{\langle F(y), y-\xbar\rangle}{\|y\|^\gamma}&\leq&   
- \lim_{k\rightarrow\infty}\frac{\langle F(y^k), \xbar-y^k\rangle}{\|y^k\|^\gamma}\leq 0
\end{eqnarray*}
which is a contradiction to the coercivity of $F$.\\\\
For a given $\xbar\in \Omega$, since $\Omega$ is closed, either the supremum in $G(\xbar)$ is achieved 
at some point $\ybar\in\Omega$, or it is achieved in the limit at some point in the asymptotic cone of $\Omega$.  
In the former case there is nothing to prove.  Assume, therefore that  there exists a sequence $(y^k)_{k\in\Nbb}$
on $\Omega$ 
with  $\|y^k\|\to\infty$ as $k\to\infty$ and 
$\lim_{k\rightarrow\infty}\langle F(y^k), \xbar-y^k\rangle=G(\xbar) <\infty$. 
This, however, contradicts the assumption that $F$ satisfies \eqref{e:coercive T}, so the supremum
must be attained on $\Omega$. 
\endproof

The next theorem is a transposition of \cite[Theorem 2.1]{FriedlanderTseng07} to the setting of 
generalized variational inequalities and illuminates the connection between exact regularization, 
the dual gap function and 
the existence of Lagrange multipliers for $\gvi(\partial\varphi, S_0)$.
\begin{thm}\label{t:solution sets}
Let $\Omega$,  $F$ satisfy Assumption \ref{hyp:F}. 
\begin{enumerate}[(i)]
\item \label{t:solution sets 1} If there exists $\varepsilonbar > 0 $ such that 
$ S_0 \cap S_{G_\varepsilonbar\varphi}  \neq \emptyset $, then $ \SGeps   \subset S_0 $
for all $ \varepsilon \in (0, \varepsilonbar)$.  If, moreover, $F$ is coercive with respect to $\Omega$, 
then $ \SGeps$ is bounded for all $ \varepsilon \in (0, \varepsilonbar)$. 

\item\label{t:solution sets 2} Let $\varphi$ satisfy Assumption \ref{hyp:GVIepsilon}.
If $ x \in \Omega $ then  $ x \in S_0 $ and $ x \not\in S_\varphi $ 
implies that $ x \not\in S_\varepsilon $ for all $ \varepsilon > 0 $.

\item\label{t:solution sets 3} Let $\varphi$ satisfy Assumption \ref{hyp:GVIepsilon}.
For any $ \varepsilon > 0 $ , $ S_0 \cap S_\varepsilon \subset S_\varphi$.

\item\label{t:solution sets 4} Let $F$, $\Omega$ and $\varphi$ together satisfy Assumptions \ref{hyp:GVIepsilon} and 
\ref{hyp:GVIvarphi} and let $F$ be coercive with respect to $\Omega$.  Then for all  $\varepsilon  > 0 $, $ S_0 \cap \SGeps  \subset S_\varphi $.

\item\label{t:solution sets 5} 
Let $\varphi$ satisfy Assumption \ref{hyp:GVIepsilon} and let $ \lambdabar \ge 0 $ be a Lagrange multiplier of 
$ \gvi(\partial \varphi, S_0)$. If 
$ \lambdabar = 0 $ then $ S_0 \cap S_\varepsilon = S_\varphi $. If $ \lambdabar > 0 $ 
and, in addition, Assumption \ref{hyp:GVIvarphi} holds for $F$ coercive with respect to $\Omega$, 
then $ S_\varphi = S_0 \cap\SGeps  $ for all 
$ \varepsilon \in  ( 0, \frac{1}{\lambdabar} ]$. 

\item\label{t:solution sets 6} Let $F$, $\Omega$ and $\varphi$ together satisfy Assumptions \ref{hyp:GVIepsilon} and 
\ref{hyp:GVIvarphi} and let $F$ be coercive with respect to $\Omega$.
Let the regularization 
parameter $ \varepsilonbar > 0 $ be such that 
$ S_0 \cap S_{G_\varepsilonbar\varphi}  \neq \emptyset $. Then $ \displaystyle\frac{1}{\varepsilonbar} $
is a Lagrange multiplier of $ \gvi( \partial\varphi, S_0) $ and $ S_0 \cap \SGeps   = S_\varphi $ 
for all $\epsilon\in (0,\epsilonbar]$ with $\SGeps = S_\varphi$ for all $\epsilon\in (0,\epsilonbar)$.
\end{enumerate}
\end{thm}

\begin{proof}
\eqref{t:solution sets 1}.  Let $ \xbar \in S_0 \cap S_{G_{\epsilonbar\varphi}}$. Since $ \xbar \in S_0 $ we have
$ G(\xbar) = 0 $ and thus $ \xbar $ minimizes the  convex function $ G $  over the convex set $ \Omega $.  
In fact the set of all minimizers of $ G $ over $ \Omega $ is exactly $ S_0$. 
Now choose any $ x \in  \Omega \setminus S_0$.  At such points we have  
$ G_{\bar{\varepsilon}} (\xbar) \le G_{\bar{\varepsilon}} (x) $ and $ G(\xbar) < G(x)$, where 
$G_{\bar{\varepsilon}} (x) = G(x) + \bar{\varepsilon} \varphi (x)$.  
Let $ \varepsilon \in (0, \bar{\varepsilon})$ and note that 
\begin{eqnarray*}
\frac{\varepsilon}{\bar{\varepsilon}} G_{\bar{\varepsilon}} (\xbar) = 
\frac{\varepsilon}{\bar{\varepsilon}} G_{\bar{\varepsilon}} (\xbar) + (1 - \frac{\varepsilon}{\bar{\varepsilon}})  G(\xbar),
\end{eqnarray*}
and, for any $y\in \Omega$, 
\[
\frac{\varepsilon}{\bar{\varepsilon}} G_{\bar{\varepsilon}} (y) + (1 - \frac{\varepsilon}{\bar{\varepsilon}})  G(y)  = G_\varepsilon (y) .
\]
Since $0 < \displaystyle\frac{\varepsilon}{\bar{\varepsilon}} < 1$, this yields,  for $ x \in \Omega \setminus S_0$,
\begin{eqnarray*}
G_\varepsilon (\xbar)  <  \frac{\varepsilon}{\bar{\varepsilon}}  G_{\bar{\varepsilon}} (x) + (1 - \frac{\varepsilon}{\bar{\varepsilon}}) G(x) =G_\epsilon(x) ,
\end{eqnarray*}
thus $ x \not\in \SGeps$. By contraposition we have 
$ x \in \SGeps$ for $\varepsilon \in (0, \bar{\varepsilon})$ implies $ x \in S_0 $.  This yields the first statement.  If, in addition 
$F$ is coercive, by Proposition \ref{t:existence}\eqref{t:existence i}, $S_0$ is bounded, hence $\SGeps$ is bounded for all 
$\varepsilon \in (0, \bar{\varepsilon})$.
\hfill$\triangle$\\

\eqref{t:solution sets 2}. Let $ x \in S_0\setminus S_\varphi$.  For each $ v \in \partial \varphi (x) $,
there exists $ y \in S_0 $ (depending on $ v$) such that 
\begin{eqnarray*}
\langle v, y - x \rangle < 0 .
\end{eqnarray*}
On the other hand, for the same pair $y$ and $x$, since $ y \in S_0$, we have 
\begin{eqnarray*}
\langle F(y), y - x \rangle \le 0.
\end{eqnarray*}
Since $ F $ is monotone this implies that 
\begin{eqnarray*}
\langle F(x), y - x \rangle \le 0.
\end{eqnarray*}
Then for any $ \varepsilon > 0  $ we have
\begin{eqnarray*}
\langle F(x) + \varepsilon v, y - x \rangle < 0.
\end{eqnarray*}
Hence $ x \not\in S_\varepsilon $ as claimed.
\hfill$\triangle$\\

\eqref{t:solution sets 3}. Let $ \xbar \in S_0 \cap S_\varepsilon $.  
Then  for some $ v \in \partial \varphi(\xbar) $ , we have
\begin{equation} \label{neweq1}
\langle F(\xbar) + \varepsilon v , x - \xbar \rangle \ge 0 , \quad \forall x \in \Omega 
\end{equation}
and 
\begin{eqnarray*}
\langle F(\xbar), x - \xbar \rangle \ge 0 \quad \forall x \in \Omega.
\end{eqnarray*}
On other hand for any $ x \in S_0 $ we have 
\begin{eqnarray*}
\langle F(x), \xbar - x \rangle \ge 0.
\end{eqnarray*}
Now by the monotonicity of $ F $ we have 
\begin{eqnarray*}
\langle F(\xbar), \xbar - x \rangle \ge \langle F(x), \xbar - x \rangle \ge 0.
\end{eqnarray*}
Hence $ \langle F(\xbar), x - \xbar \rangle = 0 $. Thus using (\ref{neweq1}) we conclude that
there exists $ v \in \partial \varphi (\xbar ) $ such that for all $ x \in S_0 $ we have $ \langle v, x - \xbar \rangle \ge 0$.
In other words, $ \xbar \in S_\varphi $, as claimed.\hfill$\triangle$\\

\eqref{t:solution sets 4}. Let $x_\epsilon \in S_0 \cap \SGeps  $.  For all $ x \in \Omega $,
\begin{eqnarray*}
G(x_\epsilon) + \varepsilon \varphi (x_\epsilon)  \le G(x) + \varepsilon  \varphi (x).
\end{eqnarray*}
Since $ \varepsilon > 0 $ we have, for all $ x \in \Omega$,
\begin{eqnarray*}
\frac{1}{\varepsilon} G(x_\epsilon) + \varphi(x_\epsilon) \le \frac{1}{\varepsilon} G(x) + \varphi (x).
\end{eqnarray*}
This shows that $ x_\epsilon $ solves 
\begin{eqnarray}\label{e:problah}
 \underset{\substack{x \in \Omega}}{\mbox{minimize }} \varphi(x) + \frac{1}{\varepsilon} G(x).
\end{eqnarray}
\noindent By Proposition \ref{t:existence}\eqref{t:existence ii}, $ G $ is finite-valued on $\Omega$ since $F$ is coercive
on an open set that contains $\Omega$, hence, in particular, $\DRLdom G\supset\Omega$.  
The first-order optimality conditions for \eqref{e:problah} are
\begin{eqnarray*}
0 \in \partial ( \varphi + \frac{1}{\varepsilon} G) (x_\epsilon) + N_\Omega(x_\epsilon).
\end{eqnarray*}
By Assumption \ref{hyp:GVIvarphi}, $0\in\rint\left(\DRLdom\partial \varphi - \Omega\right)\subset 
\rint\left(\DRLdom\partial \varphi - \DRLdom G\right)$ so we may apply the sum rule for subdifferentials
(see, for example, \cite[Theorem 3.39]{Penot}) for the equivalent inclusion
\begin{eqnarray*}
0 \in \partial \varphi (x_\epsilon) + \frac{1}{\varepsilon} \partial G(x_\epsilon)  + N_\Omega(x_\epsilon).
\end{eqnarray*}
As $\frac{1}{\varepsilon}\in N_{\Rbb_-}(G(x_\epsilon))$, the above inclusion is just \eqref{e:Lagrange optimality}, hence 
$\displaystyle\frac{1}{\varepsilon}$ is a Lagrange multiplier of \eqref{e:PphiG} paired with the solution
$x_\epsilon$.  Since the solution set to \eqref{e:PphiG} coincides with $S_\varphi$,  
this completes the proof of part \eqref{t:solution sets 4}.\hfill$\triangle$\\

\eqref{t:solution sets 5}.  Suppose that $\xbar \in S_\varphi $ is a solution, paired with the 
Lagrange multiplier $ \lambdabar \ge 0$, to $\gvi( \partial\varphi, S_0) $.  That is, by 
\eqref{e:Lagrange optimality} the pair
$(\xbar,\lambdabar)$ satisfies 
\begin{eqnarray}\label{e:foc GVIphi}
0 \in \partial\varphi(\xbar) + \lambdabar \partial G(\xbar)  + N_\Omega(\xbar) .
\end{eqnarray}

We consider first the case $ \lambdabar = 0 $.   The optimality condition \eqref{e:foc GVIphi}
then simplifies to  
\begin{eqnarray*}
0 \in \partial\varphi(\xbar) + N_\Omega(\xbar),
\end{eqnarray*}
\noindent hence there exists $ v \in \partial\varphi(\xbar) $ such that 
\begin{equation}\label{neweq2}
\langle v, x - \xbar \rangle \ge 0 , \quad \forall x \in \Omega .
\end{equation}
Moreover, since $ \xbar \in S_\varphi$, we know that $ \xbar \in S_0 $ and hence
\begin{equation}\label{neweq3}
\langle F(\xbar),  x - \xbar \rangle\geq 0 \quad \forall x \in \Omega.
\end{equation}
Thus multiplying \eqref{neweq2} by $ \varepsilon > 0 $ and adding to
\eqref{neweq3} yields
\begin{eqnarray*}
\langle F(\xbar) + \varepsilon v, x - \xbar \rangle \ge 0 \quad \forall  x \in \Omega,
\end{eqnarray*}
\noindent that is,  $ \xbar \in S_\varepsilon $ and hence  
$ S_\varphi \subseteq S_\varepsilon \cap S_0$. Now by Part \eqref{t:solution sets 3} we conclude that,
for $ \lambdabar = 0$,  we have $ S_\varphi = S_\varepsilon \cap S_0 $.

Consider next the case $ \lambdabar > 0$.  
Note that $\DRLdom  \varphi \cap \DRLdom  G \neq \emptyset $ since $ S_0 \subset\DRLdom  \varphi \cap \DRLdom  G$. 
Further $ \varphi $ is continuous on $ \intr \DRLdom \varphi $ and thus continuous on $ S_0 $.
By Proposition \ref{t:existence} and Assumption \ref{hyp:GVIvarphi} we can again apply the sum rule to yield
\begin{eqnarray*}
0 \in \partial\varphi(\xbar) + \lambdabar \partial G(\xbar)  + N_\Omega(\xbar) = \partial ( \varphi + \lambdabar G)(\xbar) + N_\Omega( \xbar).
\end{eqnarray*}
We conclude that $ \xbar $ is a minimizer of the convex optimization problem
\begin{eqnarray*}
\min_{x \in \Omega} \varphi (x) + \lambdabar G(x)
\end{eqnarray*}
and, hence,
\begin{eqnarray*}
\frac{1}{\lambdabar}  \varphi(\xbar) + G(\xbar) \le \frac{1}{\lambdabar} \varphi(x) + G(x), \quad \forall x \in \Omega.
\end{eqnarray*}
Now since, $ \xbar \in S_0 $, we have, in fact, $ G(\xbar) = 0$, so the above inequality simplifies to 
\begin{equation}\label{neweq5}
\frac{1}{\lambdabar}  \varphi(\xbar)  \le \frac{1}{\lambdabar} \varphi(x) + G(x), \quad \forall x \in \Omega.
\end{equation}
Also note that for any $ x \in \Omega $
\begin{equation} \label{neweq6}
0 \le G(x).
\end{equation}
\noindent Multiplying  (\ref{neweq5} ) by $ \eta $ and (\ref{neweq6}) by $( 1 - \eta) $ with $ \eta \in (0,1] $
and adding yields
\begin{eqnarray*}
\frac{\eta}{\lambdabar}  \varphi(\xbar) \le \frac{\eta}{\lambdabar} \varphi(x) + G(x), \quad \forall x \in \Omega.
\end{eqnarray*}
Again using the fact that $ G(\xbar) = 0$, the above inequality can be written as
\begin{equation}\label{neweq7}
\frac{\eta}{\lambdabar}  \varphi(\xbar)  + G(\xbar)  \le \frac{\eta}{\lambdabar} \varphi(x) + G(x), \quad \forall x \in \Omega.
\end{equation}
For all $ \varepsilon\in (0,\frac{1}{\lambdabar}]$ there is an $\eta \in (0, 1] $ with 
$ \varepsilon  = \displaystyle \frac{\eta}{\lambdabar} $.  Then, by \eqref{neweq7}, 
for all $ \varepsilon \in (0,  \frac{1}{\lambdabar}] $,
\begin{eqnarray*}
\varepsilon  \varphi(\xbar)  + G(\xbar)  \le \varepsilon \varphi(x) + G(x)  \quad \forall ~x \in \Omega. 
\end{eqnarray*}
Hence, for all $ \varepsilon \in (0,  \frac{1}{\lambdabar}] $, $ \xbar \in \SGeps$, and thus $ \xbar \in S_0 \cap \SGeps$. 
This establishes the inclusion $S_\varphi \subseteq S_0 \cap \SGeps$.  Now by part \eqref{t:solution sets 4} 
this implies that $ S_\varphi = S_0 \cap \SGeps$, as claimed.
\hfill$\triangle$\\

\eqref{t:solution sets 6}.  Suppose that  there exists $\bar{\varepsilon}  > 0 $ such that 
$ S_0 \cap S_{G_\varepsilonbar\varphi} \neq \emptyset $.
Choose $ \xbar \in S_0 \cap S_{G_\varepsilonbar\varphi}$.  Since $ \xbar \in  S_{G_\varepsilonbar\varphi}$ we have 
\begin{eqnarray*}
G(\xbar) + \bar{\varepsilon} \varphi (\xbar) \le G(x) + \bar{\varepsilon}  \varphi(x) \quad \forall x \in \Omega,
\end{eqnarray*}
and hence
\begin{eqnarray*}
\frac{1}{ \bar{\varepsilon}} G(\xbar) + \varphi(\xbar) \le \frac{1}{\bar{\varepsilon}} G(x) + \varphi(x) \quad \forall x \in \Omega.
\end{eqnarray*}
Thus $ \xbar $ solves the convex optimization problem.
\begin{eqnarray*}
\min_{x \in \Omega}  \varphi (x) + \frac{1}{\bar{\varepsilon}} G(x).
\end{eqnarray*}
Since $F$ is coercive we may apply By Proposition \ref{t:existence} to conclude that $ G $ is a 
finite convex function and $\xbar$ satisfies
\begin{eqnarray*}
0 \in \partial(\varphi + \frac{1}{\bar{\varepsilon}} G) (\xbar) + N_\Omega(\xbar).
\end{eqnarray*}
Thus using the sum rule we obtain that 
\begin{eqnarray*}
0 \in \partial\varphi (\xbar) + \frac{1}{\bar{\varepsilon}} \partial G(\xbar) + N_\Omega(\xbar).
\end{eqnarray*}
This shows that $ \frac{1}{\bar{\varepsilon}} > 0 $ is a Lagrange multiplier of the problem \eqref{e:PphiG}. 
Thus  using \eqref{t:solution sets 5} we conclude that
$ S_0 \cap \SGeps  = S_\varphi $ for all $ \varepsilon \in (0, \bar{\varepsilon}] $.  By part \eqref{t:solution sets 1},
$\SGeps\subset S_0$ for all $\epsilon\in (0,\epsilonbar)$, which yields the second statement and completes the proof. 
\hfill$\triangle$\\
\end{proof}

\begin{cor}[boundedness of solutions to \eqref{e:PGe}]\label{t:boundedness for VIeps}
Let Assumption \ref{hyp:F} hold and let
$F$ be coercive with respect to $\Omega$.  Assume further that there exists $\epsilonbar>0$ 
such that $S_0\cap S_{G_{\epsilonbar\varphi}}\neq \emptyset$.
Denote by $\Ucal_{\epsilon'}$ the set
\[
\Ucal_{\epsilon'}\equiv   \bigcup_{0< \epsilon\leq \epsilon'} S_{G_{\epsilon\varphi}}.
\]
For all  $\epsilon'<\epsilonbar$, the set  $\Ucal_{\epsilon'}$ is nonempty and bounded.
 \end{cor}
\begin{proof}  This is a direct consequence of Theorem \ref{t:solution sets}  \eqref{t:solution sets 1}.
\end{proof}

\section{Convergence of regularized VI: regularizing the dual gap function $G$ with $\varphi$}%
\label{s:reg G}
In this section we briefly discuss the solution strategies for the regularization approach given by 
\eqref{e:PGe}; that is, we regularize the dual gap function $G$ of $\vi(F,\Omega)$ 
by $\epsilon_k\varphi$ and examine solutions $x_{\epsilon_k}$ to \eqref{e:PGe} 
with parameter $\epsilon_k$. Abstractly, this simply concerns regularization of convex 
optimization problems, and therefore is well understood.   Our primary interest here is 
what relation the sequence of solutions to the regularized optimization problems has to 
the solution set to the unregularized monotone variational inequality. 
If the condition for the exact regularization 
(Theorem \ref{t:solution sets}\eqref{t:solution sets 1}) holds, then the regularized solutions, $x_{\epsilon_k}$,
lie in the solution set $S_0$ for all $k$ such that $\epsilon_k<\epsilonbar$.  Moreover, if $F$ is coercive, 
then by Corollary \ref{t:boundedness for VIeps} the sequence $(x_{\epsilon_k})$ has cluster points, all of which 
are solutions to $\vi(F, \Omega)$.  Therefore, for some $k$ large enough, in order to solve 
$\vi(F, \Omega)$ for $F$ monotone, it suffices to solve \eqref{e:PGe} for $\epsilon_k$.    
\begin{propn}\label{p:exactsolutionsgap}
 Suppose $\Omega$, $F$ satisfy Assumptions \ref{hyp:F}, and 
 let $F$ be coercive with respect to  $\Omega$.
Let $(\epsilon_k)_{k\in\Nbb}$ be a decreasing sequence on $\DRLRp$ with $\epsilon_k \searrow 0$ and let $x_{\epsilon_k}$ solve 
\eqref{e:PGe} with parameter $\epsilon_k$ for each $k\in \Nbb$.  If there exists $\epsilon > 0$ such that $S_0 \cap \SGeps \neq \emptyset$, 
then the sequence $(x_{\epsilon_k})_{k\in\Nbb}$ 
is bounded and, for all $k$ large enough,  $x_{\epsilon_k}\in S_0$.
\end{propn}
\proof
Boundedness of the sequence $(x_{\epsilon_k})_{k\in\Nbb}$ follows from Corollary \ref{t:boundedness for VIeps}.  Indeed, since 
$S_0$ is bounded (Lemma \ref{t:S_0 compact}) with $S_{G_{\epsilon_k\varphi}}\subset S_0$ for all 
$\epsilon_k\in(0, \epsilonbar)$ (Theorem \ref{t:solution sets}\eqref{t:solution sets 1}), then the result follows immediately.  
\endproof

Motivated by the study of error bounds in \cite{FriedlanderTseng07}, we now derive an error bound
for $d(S_0,\SGeps)$ in a analogous framework to \cite[Theorem 5.1]{FriedlanderTseng07}.  
For this we introduce
 the concept of {\em weak sharpness of order $\gamma>1$} for the solution sets of variational inequalities.

 The notion of weak sharp minimum for a convex minimization problem  has been introduced by Burke and 
Ferris \cite{BurkeFerris93}.  We recall that the solution set $S_f\equiv \argmin_{x\in
\Omega}\{f(x)\}$ is weakly sharp if there exists a positive number $\alpha$ (sharpness constant) such that 
\begin{equation*}
 f(x)\geq f(\xbar)+\alpha~d(x, S_f)~~~\forall \xbar\in S_f.
\end{equation*}
Similarly, the solution set $S_f$ is weakly sharp 
of order $\gamma$ if there exists a positive number $\alpha$ (sharpness constant) such that, 
for each $x\in\Omega$, 
\begin{equation*}
 f(x)\geq f(\xbar)+\alpha~d(x, S_f)^\gamma~~~\forall \xbar\in S_f.
\end{equation*}

For any $A\subset \mathbb{R}^n$, it's polar cone is defined as 
$A^\circ=\{y \in \mathbb{R}^n:\langle y, x \rangle \leq 0\quad  \forall x\in A\}$. The
relationship between a cone and its polar cone is, similar to that between a linear subspace and its orthogonal
complement. From the characterization of a weak sharp solution for a convex minimization problem with a closed proper
objective function $f$,
Marcotte and Zhu \cite{MarcotteZhu98} extended the concept
of weak sharp minima for the variational inequality problem. The solution set $S_0$ of 
$\vi(F,\Omega)$ is weakly sharp if, for any $\xbar\in S_0$,
\begin{equation}\label{weaksharpVI}
 -F(\xbar)\in \intr\left(\bigcap_{x\in S_0}[T_\Omega(x)\cap N_{S_0}(x)]^\circ\right).
\end{equation}
 However, it is not obvious how to extend \eqref{weaksharpVI} for orders $\gamma>1$. Since the dual gap 
 function $G(x)$ casts $\vi(F, \Omega)$ as a convex minimization problem, an alternative notion of weak
 sharp minima of a variational inequality of order 1  based on $G(x)$ has been proposed in \cite{MarcotteZhu98}. 
We extend this to orders $\gamma>1$ and propose a generalization. That is, the set $S_0$ is weakly sharp of 
order $\gamma>1$ if there exists a positive number
$\alpha$ (the sharpness constant) such that 
\begin{equation}\label{weaksharpVIG-gamma}
G(x)\geq \alpha~d(x,S_0)^\gamma~~~\forall x\in \Omega.
\end{equation}
 \begin{thm}\label{t:theorem-errorbound-ereg}
Let $F$, $\Omega$ and $\varphi$ together satisfy Assumptions \ref{hyp:F}, \ref{hyp:GVIepsilon} and 
\ref{hyp:GVIvarphi}, and let $F$ be coercive with respect to $\Omega$. 
Suppose that the solution set $S_0$ is weakly sharp of order $\gamma> 1$ with sharpness constant
$\alpha>0$.
Then there exists $\tau>0$ such that, for all $\epsilon>0$, 
  \begin{equation}
   d(x_\epsilon, S_0)^{\gamma-1}\leq \tau\epsilon~~~\forall~~x_\epsilon
\in \SGeps.
  \end{equation}
In particular, $\SGeps$ is bounded for each $\epsilon>0$.   
 \end{thm}
\begin{proof}
 Let $x_\epsilon\in  \SGeps$ for some $\epsilon>0$ and let $\xbar_\epsilon=P_{S_0}(x_\epsilon)$, the projection 
being nonempty by Assumption \ref{hyp:GVIvarphi}. Then, from the definition
 of weak sharp minima
 \[
G(\xbar_\epsilon)+\epsilon\varphi(\xbar_\epsilon)\geq G(x_\epsilon)+\epsilon\varphi(x_\epsilon)\geq \alpha d(x_\epsilon,
S_0)^\gamma+\epsilon\varphi(x_\epsilon).
\]
Note that, $G(\xbar_\epsilon)=0$, hence 
\begin{equation}\label{e:subdifferentialphi}
 \alpha d(x_\epsilon,S_0)^\gamma= 
\alpha \|x_\epsilon-\xbar_\epsilon\|^\gamma\leq\epsilon(\varphi(\xbar_\epsilon)-\varphi(x_\epsilon)).
\end{equation}
From the definition of the subdifferential of a convex, real-valued map $\varphi$, we have
\[
\varphi(\xbar_\epsilon)-\varphi(x_\epsilon)\leq \langle v_\epsilon, \xbar_\epsilon-x_\epsilon \rangle\leq 
\|v_\epsilon\|\|x_\epsilon-\xbar_\epsilon\|~~\mbox{for all}~~v_\epsilon\in\partial\varphi(\xbar_\epsilon),
\]
thus, it follows from \eqref{e:subdifferentialphi} that
\begin{equation}
 \alpha \|x_\epsilon-\xbar_\epsilon\|^{\gamma-1}\leq\epsilon\|v_\epsilon\|.
\end{equation}
Now, for $F$ and $\Omega$ satisfying Assumption \ref{hyp:F} with 
$F$ coercive on $\Omega$, the solution set $S_0$ is bounded (Proposition \ref{t:existence}\eqref{t:existence i}). 
Moreover, by Assumption \ref{hyp:GVIepsilon},
$\varphi$ is convex and continuous on $\Omega$, and hence convex and continuous on $S_0$.   Consequently,  
$\partial\varphi$, that is $\|v_\epsilon\|$, is bounded, uniformly, on the compact set $S_0$.   
% \cite[Theorem 24.7]{CA},
Hence the statement follows  with $\tau'=\alpha^{-1}M$, where $M$ is the uniform bound for
$\|v\|$ with $v\in\partial\varphi(S_0)$.
\end{proof}

\noindent Note that this error bound is independent of the existence of Lagrange multipliers or
the coincidence of the solution sets $S_0$ and $\SGeps$ for some $\epsilon$ 
(Theorem \ref{t:solution sets}\eqref{t:solution sets 6}).

\section{Convergence of regularized VI: regularizing $F$ with $\nabla\varphi$}%
\label{s:reg VI}
In this section we study the other case of the regularization, where we
solve $\vi(F,\Omega)$ through a sequence of regularized problems $\vi(T_\epsilon, \Omega)$, 
where $T_\epsilon=F+\epsilon \nabla\varphi$. We restrict ourselves to a differentiable regularization to make the computation easier.
 We are interested in the approximate solutions to $\vi(T_\epsilon, \Omega)$ in view of the algorithm that we present later 
in this section which, in principle, 
 involves generating a sequence of solutions to the regularized problems $VI(T_{\epsilon}, \Omega)$ as $\epsilon\rightarrow 0$.
 Since, it is not possible to compute the exact solution to $VI(T_{\epsilon}, \Omega)$ in practice,
 we seek an approximate solution to $VI(T_{\epsilon}, \Omega)$ for every $\epsilon>0$ with some error tolerance.  Knowing 
that we are within a given error tolerance is the chief concern of {\em error bounds}, which we determine in 
Proposition \ref{t:errorbound}.  Error bounds between points in $S_\epsilon$ and $S_0$ are discussed briefly in section
\ref{s:error bounds S_eps}.

\subsection{Convergence of regularized solutions}
We begin with a 
study of the behavior of the path $\{x_\epsilon:\epsilon>0\}$ where $x_\epsilon$ is the unique solution to $VI(T_{\epsilon}, \Omega)$ and proceed to
 show that all the cluster points of the sequences of 
solutions  (exact or approximate)  to $\vi(T_\epsilon, \Omega)$
 are the solutions to $\vi(F, \Omega)$ as $\epsilon\rightarrow 0$. 
\begin{thm}
 When $\varphi$ is strongly convex and Fr\'echet differentiable and $F$ is coercive, 
then the map $\epsilon \mapsto x_\epsilon$ is continuous.
\end{thm}
\proof Let $x_\epsilon$ solve $\vi(T_\epsilon, \Omega)$. Since $\varphi$ is strongly 
convex, $\nabla\varphi$ is strongly monotone and hence
$T_\epsilon=F+\epsilon\nabla\varphi$ is strongly monotone. Hence, there exists a $\mu_\epsilon>0$
such that for any 
$x,y \in \Omega$, 
\begin{equation*}
 \langle T_\epsilon (x)-T_\epsilon (y), x-y \rangle \geq \mu_\epsilon \|x-y\|^2.
\end{equation*}This implies
\begin{equation}\label{e:reggap and dual gap}
 \sup_{y\in \Omega} \left[ \langle T_\epsilon (x), x-y \rangle- \mu_\epsilon \|x-y\|^2
\right]\geq \sup_{y\in \Omega}\langle T_\epsilon (y), x-y \rangle.
\end{equation}
The expression on the left side of \eqref{e:reggap and dual gap} is the regularized gap function $\theta(\cdot,\epsilon\varphi)$,
which is zero at $x=x_\epsilon$. The right hand side is the dual gap function
for $\vi(T_\epsilon , \Omega)$ (see \eqref{e:dual gap func}), which we denote by $G(x,\epsilon\varphi)$.
Since $ G(x;\epsilon\varphi)\geq 0$ for all $x\in \Omega$, \eqref{e:reggap and dual gap} implies that 
$ G(x_\epsilon;\epsilon\varphi)=0$ and hence $x_\epsilon \in S_{G(.;\epsilon\varphi)}$,  where $S_{G(.;\epsilon\varphi)}$
is the solution set of the convex minimization problem $\min_{x\in \Omega}G(.;\epsilon\varphi)$. 
Now by Proposition \ref{t:existence}\eqref{t:existence ii}, $G$ is finite valued, and hence $G$ 
continuous (since it is convex \cite[Theorem 10.1]{CA}) on the relative interior of $\Omega$, which is nonempty  
 as $\Omega$ is nonempty \cite[Theorem 6.2]{CA}.
Therefore, the map $\epsilon \mapsto S_{G(.;\epsilon\varphi)}$ is upper-semicontinuous as a set-valued map 
in the sense of \cite[Theorem 4.3.3]{BankGuddatKlatteKummerTammer}. 
However, the strong monotonicity of $T_\epsilon$ implies that $S_\epsilon$(=$S_{G(.;\epsilon\varphi)})$ is singleton 
and hence the map $\epsilon \mapsto x_\epsilon$ is continuous. \endproof

The type of continuity used in the above proof is not the same as outer semincontinuity defined in \ref{d:osc}, which 
is the same as graph closedness. 

Our next results are on the convergence of the sequences of  solutions (exact or approximate) to $\vi(T_\epsilon, \Omega)$.
  Before we define the concept of an approximate solution to $\vi(T_\epsilon, \Omega)$, let us introduce some notation. 
For a given $\epsilon>0$ we denote the regularized
gap function for $VI(T_{\epsilon}, \Omega)$ by $ \theta_{\alpha}(. ;\epsilon\varphi)$ and 
the D-gap function for $VI(T_{\epsilon}, \Omega)$ by $ \theta_{\alpha\beta}(. ;\epsilon\varphi)$. 
These are given by (similar to \eqref{e:reggap} and \eqref{e:Dgap})
\begin{equation}\label{e:reggap-teps}
 \theta_{\alpha}(x ;\epsilon\varphi)=\sup_{y\in \Omega} \left\{\langle T_\epsilon(x), x-y \rangle
-\frac{\alpha}{2}\|y-x\|^2\right\}
 \end{equation}
\begin{equation}\label{e:dgap-teps}
 \theta_{\alpha\beta}(x ;\epsilon\varphi)=\theta_{\alpha}(x ;\epsilon\varphi)-\theta_{\beta}(x ;\epsilon\varphi),\quad (\alpha<\beta).
 \end{equation}
We write this more succinctly using the projection. 
\begin{equation}\label{e:reggap-teps-explicit}
  \theta_{\alpha}(x;\epsilon\varphi)=\langle
T_{\epsilon}( x),
x-y^{\epsilon}_{\alpha}(x)
 \rangle-\frac{\alpha}{2}\|y^{\phi,\epsilon}_{\alpha}(x)-x_{\epsilon}\|^2,
\end{equation}
with 
\begin{equation}\label{e:yalpha epsilon}
y^{\epsilon}_{\alpha}(x)=P_\Omega[x-\frac{1}{\alpha}T_{\epsilon}( x)].
\end{equation}
The regularized gap function $\theta_{\beta}(. ;\epsilon\varphi)$ is defined analogously with, instead, the projection
$y^{\epsilon}_{\beta}(x)$. 

Recall that for any solution $\xbar_\epsilon$ of $VI(T_\epsilon, \Omega)$,
$ \theta_{\alpha\beta}(\xbar_\epsilon ;\epsilon\varphi)=0$. We define a point $x$ to
be an \textit{approximate
solution} to $VI(T_\epsilon, \Omega)$ with an error $\zeta>0$ if
\[ \theta_{\alpha\beta}(x ;\epsilon\varphi)\leq \zeta.\] 
\begin{thm}\label{t:mainconvergence_exactreg}
 Let $(\epsilon_k)_{k\in\Nbb}$ be a sequence of nonnegative scalars with $\epsilon_k \searrow 0$,
and let $(x^{k})_{k\in\Nbb} \in \Rbb^n$ be a sequence of
approximate solutions of $\vi(T_{\epsilon_k}, \Omega)$ with errors $\zeta_k\geq 0$.  
Assume that $\varphi$ is continuously differentiable.
If $x^{k}\rightarrow \bar{x}$ as $k\rightarrow \infty$ and if
$\zeta_k\searrow 0$, then $\bar{x}$ solves $\vi(F, \Omega)$.
\end{thm}
\proof Choose $0<\alpha<\beta$ and assume that $x^{k}\rightarrow \bar{x}$. Since the projection
map onto a closed convex set is continuous, we have $y^{\epsilon_k}_{\alpha}(x^{k})\rightarrow y_{\alpha}(\xbar)$ and
$y^{\epsilon_k}_{\beta}(x^{k})\rightarrow y_{\beta}(\xbar)$ as $k \rightarrow \infty$,
where $y^{\epsilon_k}_{\alpha}$ is defined by \eqref{e:yalpha epsilon}.
Using \eqref{e:reggap-teps-explicit},
\begin{equation*}
 \lim_{k\rightarrow \infty}  \theta_{\alpha}(x^{k};\epsilon_k\varphi)=\langle
F(\bar{x}), \bar{x}-y_\alpha(\bar{x})
 \rangle-\frac{\alpha}{2}\|y_\alpha(\bar{x})-\bar{x}\|^2=\theta_{\alpha}(\bar{x}
)
\end{equation*}
and
\begin{equation}
 \lim_{k\rightarrow \infty}  \theta_{\beta}(x^{k};\epsilon_k\varphi)=\langle
F(\bar{x}), \bar{x}-y_\beta(\bar{x})
 \rangle-\frac{\beta}{2}\|y_\beta(\bar{x})-\bar{x}\|^2=\theta_{\beta}(\bar{x}).
\end{equation}
Now, form \eqref{e:dgap-teps}
\begin{eqnarray}\label{e:relation of limits of dgap}
 0 \leq \lim_{k\rightarrow \infty}
 \theta_{\alpha\beta}(x^{k};\epsilon_k\varphi)&=&\lim_{k\rightarrow \infty}
 \theta_{\alpha}(x^{k};\epsilon_k\varphi)-\lim_{k\rightarrow \infty}
 \theta_{\beta}(x^{k};\epsilon_k\varphi)\\& =&\theta_{\alpha}(\bar{x}) -\theta_{\beta}(\bar{x}) =\theta_{\alpha\beta}(\bar{x}) 
\end{eqnarray}
and since $x^{k}$ is a sequence of approximate solutions
\begin{equation}\label{e:limit of theta alpha beta}
 \lim_{k\rightarrow \infty}\theta_{\alpha\beta}(x^{k};\epsilon_k\varphi)\leq  \lim_{k\rightarrow \infty}\zeta_k=0.
\end{equation}We conclude from \eqref{e:relation of limits of dgap} and \eqref{e:limit of theta alpha beta} that
 $\theta_{\alpha\beta} (\bar{x})=0$ and therefore $\bar{x}$ solves $\vi(F, \Omega)$ \cite[Theorem 10.3.3]{FacchineiPang}. \endproof\\
 As noted in the introduction, 
solving $\vi(T_{\epsilon},\Omega)$ is equivalent to minimizing the gap
function $\theta_{\alpha}(. ;\epsilon\varphi)$ over $\Omega$ or 
$\theta_{\alpha\beta}(. ;\epsilon\varphi)$ over $\Rbb^n$. If we  minimize
$ \theta_{\alpha}(. ;\epsilon\varphi)$ over $\Omega 
$ using standard optimization methods for the constrained
case, 
we will in effect generate a sequence of solutions $x_{\epsilon}$ which are
actually be considered as solutions to $\vi(T_{\epsilon}, \Omega)$.
Alternatively, if we minimize the D-gap function $ \theta_{\alpha\beta}(. ;\epsilon\varphi)$  over $\Rbb^n$,
we also generate a sequence of solutions to  $\vi(T_{\epsilon}, \Omega)$.

Using the error bounds for a strongly monotone variational inequalities, we now deduce an error bound for the 
distance between any
point and a true solution of $\vi(T_\epsilon, \Omega)$ in terms of the corresponding D-gap function
$\theta_{\alpha\beta}(.;\epsilon\varphi)$, provided that $F$ is Lipschitz continuous and $\varphi$ and strong convex.
This error bound can be used as an implementable stopping criterion for the 
 algorithms aimed at approximately solving $\vi(T_\epsilon, \Omega)$. 

The proof of following lemma 
goes along the lines the proof of Theorem 3.2 in \cite{Dutta12} adapted to $\vi(T_{\epsilon},\Omega)$.
\begin{lemma}
  Let $\varphi$ be strongly convex with modulus $\rho$, $F$ and $\nabla\varphi$ be Lipschitz on $\Omega$ with constants 
$L$ and $M$ respectively. If $x_\epsilon$ solves $\vi(T_{\epsilon},\Omega)$, then, for any $x\in\Omega$,
 \begin{equation}\label{e:duttaineq}
 \|x-x_\epsilon\|\leq
\frac{\beta+L+\epsilon M}{\epsilon\rho}\|y^{\epsilon}_\beta(x)-x\|,
\end{equation}
where $y^{\epsilon}_\beta(x)$ is the point where the supremum in $\theta_{\beta}(x_{\epsilon},\epsilon\varphi)$ is attained
and given by $y^{\epsilon}_{\beta}(x)=P_\Omega[x-\frac{1}{\beta}T_{\epsilon}( x)]$. 
\end{lemma}
\proof Since $\varphi$ is strongly convex on $\Omega$ with modulus $\rho$,
$\nabla\varphi$ is strongly monotone on $\Omega$ with modulus $\rho$. Since $F$ is monotone, 
$T_\epsilon$ is strongly monotone with modulus of strong monotonicity $\epsilon\rho$. Also,
$T_\epsilon$ is Lipschitz with constant $L+\epsilon M$. From \eqref{e:reggap-teps},
$y^{\epsilon}_\beta(x)$ maximizes the function $y \rightarrow \langle T_\epsilon(x), x-y \rangle-\frac{\beta}{2}\langle y-x,
y-x\rangle$. Hence,
$y^{\epsilon}_\beta(x)$
is the unique minimizer of the strongly convex function 
$y \rightarrow \langle T_\epsilon(x), y-x \rangle+\frac{\beta}{2}\langle y-x,
y-x\rangle$. The optimality conditions yield
\begin{equation*}
 \langle T_\epsilon(x)+\beta(y^{\epsilon}_\beta(x)-x), x_\epsilon-y^{\epsilon}_\beta(x)\rangle\geq 0.
\end{equation*}
Moreover, since $x_\epsilon$ solves $\vi(T_{\epsilon},\Omega)$, we have
\begin{equation*}
 \langle T_\epsilon(x_\epsilon), y^{\epsilon}_\beta(x)-x_\epsilon\rangle \geq 0.
\end{equation*}
Hence the two inequalities above yield
\begin{equation*}
  \langle T_\epsilon(x)-T_\epsilon(x_\epsilon)+\beta(y^{\epsilon}_\beta(x)-x), y^{\epsilon}_\beta(x)-x_\epsilon \rangle\leq 0.
\end{equation*} 
Now, 
\begin{equation*}
  \langle T_\epsilon(x)-T_\epsilon(x_\epsilon), y^{\epsilon}_\beta(x)-x_\epsilon+x-x \rangle +
\beta\langle y^{\epsilon}_\beta(x)-x, y^{\epsilon}_\beta(x)-x_\epsilon+x-x \rangle\geq 0,
\end{equation*}
which implies
\begin{eqnarray*}
 \langle T_\epsilon(x)-T_\epsilon(x_\epsilon), x-x_\epsilon \rangle &+& \beta \langle y^{\epsilon}_\beta(x)-x, y^{\epsilon}_\beta(x)-x \rangle \leq 
 -\beta\langle y^{\epsilon}_\beta(x)-x, x-x_\epsilon \rangle \\ & &- \langle T_\epsilon(x)-T_\epsilon(x_\epsilon), y^{\epsilon}_\beta(x)-x \rangle.
\end{eqnarray*}
Since $T_\epsilon$ is Lipschitz with constant $L+\epsilon M$ and strongly monotone with modulus $\epsilon\rho$, we get
\begin{equation*}
 \epsilon\rho\|x-x_\epsilon\|^2\leq\beta\|y^{\epsilon}_\beta(x)-x\|\|x-x_\epsilon\|+(L+\epsilon M)\|y^{\epsilon}_\beta(x)-x \|\|x-x_\epsilon\|.
\end{equation*}
Therefore,
\begin{equation*}
 \|x-x_\epsilon\|\leq \frac{\beta +L+\epsilon M}{\epsilon\rho}\|y^{\epsilon}_\beta(x)-x \|
\end{equation*}
as claimed.
\endproof

\begin{propn}\label{t:errorbound}
 Let $\varphi$ be strongly convex with modulus $\rho$, and let $F$ and $\nabla\varphi$ be
 Lipschitz on $\Omega$ with constants $L$ and $M$ respectively.
 If $x_\epsilon$ solves $\vi(T_{\epsilon},\Omega)$, then, for any $x\in\Omega$,
 \begin{equation}\label{e:errorbound}
 \|x-x_\epsilon\|\leq \frac{\beta+L+\epsilon M}{\epsilon\rho
}\sqrt{\frac{2}{(\beta-\alpha)}\theta_{\alpha\beta}(x;\epsilon \varphi)}.
\end{equation} 
\end{propn}
\proof Adapting \cite[Lemma 4.2, Eq(19)]{YamashitaTajiFukushima97} for $\vi(T_\epsilon,\Omega)$,
\begin{equation}\label{e:YTFineq}
 \|x-y^{\epsilon}_\beta(x)\|^2 \leq
\frac{2}{(\beta-\alpha)}\theta_{\alpha\beta}(x;\epsilon\varphi).
\end{equation}
Now, from \eqref{e:duttaineq} and \eqref{e:YTFineq} 
\begin{equation}
 \|x-x_\epsilon\|\leq  \frac{\beta+L+\epsilon M}{\epsilon\rho
}\sqrt{\frac{2}{(\beta-\alpha)}\theta_{\alpha\beta}(x;\epsilon \varphi)}
\end{equation}
 This completes the proof.
\endproof
\subsection{Sequential inexact descent method}
In this section we propose a sequential inexact descent method to solve 
the $VI(F,\Omega)$ through the regularized problems  $VI(T_\epsilon,\Omega)$ where
$T_\epsilon=F+\epsilon\nabla\varphi$. It is natural  to look for the exact 
solutions of $\vi(T_{\epsilon_k},\Omega)$, however, it is not practically 
possible to run the algorithm infinitely. We therefore must be satisfied with {\em approximation}  
of the solutions to $\vi(T_{\epsilon_k},\Omega)$ for each $k$ with an error tolerance $\tau_k$.  
Convergence behavior of the sequence of approximate solutions will then follow from Proposition \ref{t:mainconvergence_exactreg}.

Choose a starting point $x^{k,0}=\xbar^{0}$, $\epsilon_0$, $\alpha_0$ and $\beta_0$. 
We solve the unconstrained minimization problem with the objective function 
$\theta_{\alpha_k\beta_k}(.,\epsilon_k\varphi)$ for $k=0,1,2,...$. 
For each $k$, we collect the approximate solution $x^k$ and initialize the inner iteration for 
solving $\vi(T_{\epsilon_{k+1}},\Omega)$ with the point $x^{k+1,0}=x^k$. The descent method in the inner iteration
of Algorithm \ref{alg} can be chosen to be any descent method that achieves sufficient decrease
in the direction of the descent so that the convergence is guaranteed. 
The regularization parameters $\epsilon_k$ are updated so that
$\epsilon_k\rightarrow 0$ as $k \rightarrow \infty$ and the
parameters $\alpha_k$ and $\beta_k$ are updated so that $\alpha_{k+1}\geq \alpha_k$ and $\beta_{k+1}\leq \beta_k$.

\begin{algorithm}[h]
\KwData{Fix sequences of error tolerances  
$(\tau_{k})_{k \in \Nbb}$ and regularization parameters
$(\epsilon_k)_{k \in \Nbb}$ with $\epsilon_{k}\rightarrow 0$ as $k\rightarrow \infty$.
For $k=j=0$, choose the point $x^{0}$,  parameters: $\beta_0>\alpha_0>0$.
}
\For{$k=0,1,2,\dots$}{
\textbf{Inner iteration:} {\em approximately solve $\vi(T_{\epsilon_k},\Omega)$.}\\
$x^{k,0}=x^k$\\
\While{$ \|x^{k,j}- x_{\epsilon_k}\|>\tau_{k}$}{
\vskip .5 cm
Apply a descent method to the unconstrained minimization of
$\theta_{\alpha_k\beta_k}(.,\epsilon_k\varphi)$ with $x^{k,0}$ as starting point
while updating $j$.
\vskip .5 cm
}
{\bf Update:} Set $x^{k+1}=x^{k, j}$,  
choose $\alpha_{k+1},\beta_{k+1}$,
increment $k=k+1$, and reset $j=0$.
}
\caption{Sequential inexact descent algorithm}\label{alg}
\end{algorithm}
We note that many choices exist for the descent method that is used in the inner iteration of Algorithm \ref{alg}.
For example, it can be the
descent method proposed in \cite{YamashitaTajiFukushima97} which is free from calculating the derivative
of $\theta_{\alpha_k\beta_k}(.,\epsilon_k\varphi)$. Another possibility is the descent method 
in \cite{LiNg09}.  

\begin{remark}
The termination of the inner iteration requires the knowledge of
the  solution  $x_{\epsilon_k}$. It is clear that the Algorithm 1 is
implementable as long as the error estimates for $\|x^{k,j}-x_{\epsilon_k}\|$ are computable. The error bound
for $\vi(T_{\epsilon_k},\Omega)$ in Proposition \ref{t:errorbound}
is very useful to fill this gap. If $\nabla\varphi$ is $\rho$-strongly monotone, 
Lipschitz continuous over $\Omega$ with modulus $M$ and if $F$ is Lipschitz continuous
over $\Omega$ with constant $L$, then $T_{\epsilon_k}$ is strongly monotone with modulus $\epsilon_k\rho$ and Lipschitz 
with constant $L+\epsilon_k M$. Then according to 
\eqref{e:errorbound}, 
\begin{equation}\label{e:ineq_dist}
 \|x- x_{\epsilon_k}\|\leq 
L_k\sqrt{\theta_{\alpha_k\beta_k}(x,\epsilon_k\varphi)},
~\mbox{where}~ 
L_k= \frac{\beta_k+L+\epsilon_k M}{\epsilon_k\rho
}\sqrt{\frac{2}{(\beta_k-\alpha_k)}}. 
\end{equation}
Hence the stopping  criterion in the inner iteration of Algorithm \ref{alg}  can now be replaced by the
implementable rule
\begin{equation}\label{alternate_stoppingrule}
 \mbox{while}~~\theta_{\alpha_k\beta_k}(x^{k,j},\epsilon_k\varphi)>\frac{\tau_k^2}{L_k^2}
\end{equation}
since $x^{k,j}$ satisfying \eqref{alternate_stoppingrule} also satisfies $ \|x^{k,j}- x_{\epsilon_k}\|>\tau_{k}$.
\end{remark}
We now discuss the convergence of Algorithm \ref{alg} under appropriate assumptions on $F$, $\varphi$ and $\Omega$
based on the assumption that for each $k$, the descent method chosen for the inner iteration converges. 
\begin{thm}\label{t:convergenceofSIDA}
Consider the Algorithm \ref{alg} with the stopping rule replaced by the alternative stopping rule
\eqref{alternate_stoppingrule}. Assume that $F$ and $\nabla\varphi$ are Lipschitz and $\nabla\varphi$
is strongly monotone. Given a sequence of parameters $(\epsilon_k)_{k \in \mathbb{N}}$ such that
 $\epsilon_k\rightarrow 0$ as $k\rightarrow \infty$ and and a sequence
of stable error tolerances $(\tau_k)_{k \in \mathbb{N}}$, $\tau_k=\tau>0$, assume that 
for each $k$, the descent method in the inner iteration converges.
Then all
the cluster points of the sequence $(x^k)_{k \in \mathbb{N}}$ of inexact solutions generated by 
Algorithm  \ref{alg} are solutions to $\vi(F, \Omega)$.
\end{thm}
\proof 
Since our stopping rule $\theta_{\alpha_k\beta_k}(x^{k,j},\epsilon_k\varphi)>\frac{\tau_k}{L_k}$ terminates
the iterations early, the inner iteration in Algorithm \ref{alg} is an early terminated variant of the
descent method that is chosen. Hence for any fixed $k$, any accumulation point 
of the sequence $x^{k,j}$ delivers an approximate solution $x^k$ to the problem
$\vi(T_{\epsilon_k},\Omega)$. Since $x^k$ violates the stopping rule, $\theta_{\alpha_k\beta_k}(x^{k},\epsilon_k\varphi)\leq \frac{\tau_k}{L_k}$
where $L_k$ is given as in \eqref{e:ineq_dist}, and since we chose
stable error tolerances $\tau_k=\tau$, we have $\frac{\tau_k}{L_k}\searrow 0$ since $L_k\to\infty$ as $k \rightarrow \infty$.
Hence by Theorem \ref{t:mainconvergence_exactreg} all
the cluster points of $(x^k)_{k\in\Nbb}$ are solutions to $\vi(F, \Omega)$.
\endproof
\begin{remark}
 The convergence of the sequence (a subsequence if necessary) of  inexact
solutions $(x^k)_{k\in\Nbb}$ generated in the  Algorithm \ref{alg} is guaranteed 
provided $(x^k)_{k\in\Nbb}$ is bounded. We now establish the sufficient conditions 
for the boundedness  of the sequence $(x^k)_{k\in\Nbb}$ along the similar lines of
\cite{FacchineiKanzow99}, however, without using the Mountain Pass Theorem.
We need the following Lemma.
\end{remark}
\begin{lemma}\label{Uniformcont}
Let $K\subset \Rbb^n$ be compact set and let $F$ and $\nabla\varphi$ be continuous functions on
$K$. Then for any $\epsilon'>0$ the gap function
$\theta_{\alpha\beta}(.,\epsilon\varphi)$, $0<\alpha<\beta$ is uniformly continuous as a function of $(x, \epsilon)$ on $K\times [0,
\epsilon']$ . In particular, for every
$\delta>0$, there exists an $\epsilonbar>0$ such that 
 \begin{equation}
  |\theta_{\alpha\beta}(x,\epsilon\varphi)-\theta_{\alpha\beta}(x)|\leq \delta
 \end{equation}
for all $(x, \epsilon)\in K \times [0, \epsilonbar]$.
\end{lemma}

\begin{proof}
Recall the D-gap function $\theta_{\alpha\beta}({.,\epsilon\varphi})$ for $VI(T_{\epsilon}, \Omega)$ is given by
\begin{equation*}
 \theta_{\alpha\beta}(x, \epsilon\varphi)=\theta_{\alpha}
(x, \epsilon\varphi)-\theta_{\beta}(x, \epsilon\varphi)
\end{equation*}where
\begin{equation*}
\theta_{\alpha}
(x, \epsilon\varphi)=\langle
F(x)+\epsilon \nabla \varphi(x),
x-y_\alpha^{\epsilon}(x)
 \rangle-\frac{\alpha}{2}\|y^{\epsilon}_\alpha(x)-x\|^2
\end{equation*}with
\begin{equation}\label{proj}
 y^{\epsilon}_\alpha(x)=P_\Omega[x-\frac{1}{\alpha}(F(x)+\epsilon \nabla\varphi(x)].
\end{equation}
Let $(x^n,\epsilon_n)_{n\in\Nbb}$ be a sequence in $K\times \Rbb_+$ and let
$(x^n,\epsilon_n)\rightarrow (x, \epsilon)$ as $n\to\infty$. Since $F$ and $\nabla\varphi$ are continuous, 
and since the projection map on a closed convex set is continuous, we have from
\eqref{proj} that
\begin{eqnarray*}
  \lim_{n\rightarrow \infty}y^{\epsilon_n}_\alpha(x^n)&=&\lim_{n\rightarrow
\infty}P_\Omega[x^n-\frac{1}{\alpha}(F(x^n)+\epsilon_n \nabla\varphi(x^n))]\\
  &=&P_\Omega[x-\frac{1}{\alpha}(F(x)+\epsilon \nabla\varphi(x))]=y^{\epsilon}_\alpha(x).
\end{eqnarray*}
Thus $ y^{\epsilon}_\alpha(x)$ viewed as a function of $x$ and $\epsilon$ is
continuous on $K\times \Rbb_+$. This implies that the function
$\theta_{\alpha}(.;\epsilon\varphi)$ is 
continuous on $K\times \Rbb_+$ as a function of $(x, \epsilon)$ and so is 
$\theta_{\alpha\beta}(.;\epsilon\varphi)$. Since 
$K$ is a compact set, for any $\epsilon'\in \Rbb_+$,
$\theta_{\alpha}(x;\epsilon\varphi)$ is uniformly continuous on $K\times [0,
\epsilon']$. In particular, for a fixed $x\in K$, it holds that for any $\delta>0$, there exists a $0<\epsilonbar<\epsilon'$ such
that for every $\epsilon \in [0, \epsilonbar]$
\begin{equation*}
 | \theta_{\alpha\beta}(x;\epsilon\varphi)- \theta_{\alpha\beta}(x)|\leq \delta.
\end{equation*}
\end{proof}
\begin{thm}
 Consider Algorithm \ref{alg} with the stopping rule
\eqref{alternate_stoppingrule}. Assume that $F$ and $\nabla\varphi$ are Lipschitz 
and $0<\alpha_k<\beta_k$ for each $k$. Assume that the solution set $S_0$ is nonempty
and bounded and that $\epsilon_k\rightarrow 0$. Then the sequence $(x^k)_{k\in\Nbb}$ generated by
the Algorithm \ref{alg} is bounded.
\end{thm}
\begin{proof} Assume that the sequence $(x^k)_{k\in\Nbb}$ 
generated by the Algorithm \ref{alg} is not bounded. Then there exists
a compact
set $K\in \Rbb^n$ such that $S_0\subset \intr K$ and $x^k \not\in K$
for
sufficiently large $k$.
Denote
\begin{equation}
\mbar_k:=\min_{x\in \partial K}\theta_{\alpha_k\beta_k}(x),
\end{equation}
where we use $\partial K$ to denote the boundary of $K$ (not to be confused with the 
subdifferential, though this should be clear from context).
Since the gap function $\theta_{\alpha_k\beta_k}$ is non-negative on $\Rbb^n$ and
since 
$S_0\subset \intr K$,
it is clear that $\mbar_k>0$.
Since  $\theta_{\alpha_k\beta_k}(x)\geq \mbar_k$ for any $x\in \partial K$, choosing
$\delta=c\mbar_k, c\in (0,1)$ we have from Lemma  \ref{Uniformcont} that
\begin{equation*}
 \theta_{\alpha_k\beta_k}(x;\epsilon_k\varphi) \geq
\theta_{\alpha_k\beta_k}(x)-c\mbar_k \geq
\mbar_k-c\mbar_k=(1-c)\mbar_k~~\forall x\in \partial K,
\end{equation*} 
which implies that
\begin{equation} \label{m-and-mbar}
 m_k:=\min_{x\in \partial K}\theta_{\alpha_k\beta_k}(x;\epsilon_k\varphi)\geq
(1-c)\mbar_k.
\end{equation}
Let $\xbar\in S_0$. Then $\theta_{\alpha_k\beta_k}(\xbar)=0$ and hence, again from Lemma  \ref{Uniformcont},
\begin{equation}\label{firstineq}
\theta_{\alpha_k\beta_k}(\xbar;\epsilon_k\varphi)=\theta_{\alpha_k\beta_k}(\xbar;\epsilon_k\varphi)-
\theta_{\alpha_k\beta_k}(\xbar) \leq c\mbar_k.
\end{equation}
Since $\theta_{\alpha_k\beta_k}(x^k;\epsilon_k\varphi)\leq
\frac{\tau_k^2}{L_k^2}$ by the
stopping rule \eqref{alternate_stoppingrule}, and since $\frac{\tau_k^2}{L_k^2}\to 0$, for sufficiently large
$k$, we have
\begin{equation}\label{secondineq}
 \theta_{\alpha_k\beta_k}(x^k;\epsilon_k\varphi)\leq c\mbar_k.
\end{equation}
Let $k$ be sufficiently large such that $x^k\not\in K$ and the inequalities
\eqref{m-and-mbar}-\eqref{secondineq}
hold. Since $c\in (0,1)$, from \eqref{m-and-mbar} we have
$c\mbar_k\leq\frac{c}{1-c}m_k<m_k$.
  
Without loss of generality assume that
$\theta_{\alpha_k\beta_k}(x^k;\epsilon_k\varphi)\leq
\theta_{\alpha_k\beta_k}(\xbar;\epsilon_k\varphi)$. 
From Weierstrass' extremal value theorem, $\theta_{\alpha_k\beta_k}(.;\epsilon_k\varphi)$
must attain a maximum at least once in $[x^k, \xbar]$. 
Let $\xhat_k\in [x^k, \xbar]$ be the point where
$\theta_{\alpha_k\beta_k}(.;\epsilon_k\varphi)$ attains its maximum. Now, viewing
$\xhat_k$ 
as a local
maximizer, it satisfies \cite[Proposition 2.3.2]{Clarke83}
\begin{equation}
 0=\nabla\theta_{\alpha_k\beta_k}(\xhat_k;\epsilon_k\varphi).
\end{equation}
Since $\xbar\in \intr K$ and $x^k\not\in K$, 
there exists a $\lambda\in(0,1)$ such that $x^k_\lambda=\lambda x^k+(1-\lambda)
\xbar \in \bdy K\cap [x^k, \xbar]$. Now 
$\theta_{\alpha_k\beta_k}(\xhat_k;\epsilon_k\varphi)\geq
\theta_{\alpha_k\beta_k}(x_\lambda^k;\epsilon_k\varphi)\geq m_k$. Hence 
$\theta_{\alpha_k\beta_k}(\xhat_k;\epsilon_k\varphi) > 0$. But, the stationary
point
$\xhat_k$ must be a
global minimizer of the D-gap function
$\theta_{\alpha_k\beta_k}(.;\epsilon_k\varphi)$ \cite[Theorem
4.3]{LiNg09}, which is a contradiction.
\end{proof}

\subsection{Error bounds}\label{s:error bounds S_eps}
Our goal in this section is to develop error bounds for the distance between the solution sets
$S_\epsilon$ and $S_0$. In \cite[Theorem 4.1]{MarcotteZhu98} it is shown that, if the solution set  $S_0$
of $VI(F,\Omega)$ is weakly sharp, $\Omega$ is compact and $F$ is pseudomonotone$^+$, then
there exists a positive number $\alpha$ such that
     \begin{equation}\label{weaksharpimplicationcompact}
    G(x)\geq \alpha~d(x, S_0)\quad \forall x\in \Omega.
     \end{equation}
     Hence under these three assumptions, we can have one type of error bound in terms of the dual gap function $G$ for the 
     distance between $S_\epsilon$ and $S_0$. That is, for any $x_\epsilon\in S_\epsilon$ 
\[
       G(x_\epsilon)\geq \alpha~d(x_\epsilon, S_0)\quad \forall x\in \Omega.
\]
We show that, even in the absence of compactness on $\Omega$ and the pseudomonotone$^+$ property on $F$,
we can derive an error bound for $d(x_\epsilon, S_0)$. 

In the proof  of \cite[Theorem 4.1]{MarcotteZhu98}, it is shown that 
when the solution set  $S_0$ of $VI(F,\Omega)$ is weakly sharp, 
that is, if \eqref{weaksharpVI} holds for $S_0$,
then there exists a positive number $\alpha$ such that for any $x\in \Omega$ and $\xbar=P_{S_0}(x)$
\begin{equation}\label{weaksharpimplication}
      \langle F(\xbar), x-\xbar\rangle\geq \alpha~ d(x,S_0 ).
     \end{equation}
We use this fact to construct an error bound along the lines of Theorem \ref{t:theorem-errorbound-ereg}. 
We need the following property, which is a stronger condition than \eqref{weaksharpVIG-gamma} and is, 
to our knowledge, new.
\boxedtxt{There exist
$\alpha>0$ and $\gamma\geq 1$ such that 
\begin{equation}\label{weaksharp-gamma-stronger}
 \langle F(y), x-y \rangle \geq \alpha~d(x, S_0)^\gamma~~~\forall x\in \Omega,
y=P_{S_0}(x).
\end{equation}}
 \begin{thm}\label{t:error bounds Se}  Let $\Omega$ and $F$ satisfy Assumption \ref{hyp:F} and 
Assumption \ref{hyp:GVIvarphi}\eqref{hyp:1}, and let the function $\varphi$ be convex and differentiable
on $\Omega$.
 \begin{enumerate}[(i)]
  \item\label{t:error bounds Se i}  Assume that the solution set $S_0$ is weakly sharp (satisfies \eqref{weaksharpVI}) 
with sharpness constant
$\alpha$.  Then, for any  $\epsilon > 0$, 
and any $x_\epsilon\in S_\epsilon$,
 \begin{equation}\label{errorboundorder1}
  \alpha~dist(x_\epsilon, S_0)\leq \epsilon
(\varphi(\xbar_\epsilon)-\varphi(x_\epsilon)),
 \end{equation}
where $\xbar_\epsilon=P_{S_0}(x_\epsilon)$.
  \item\label{t:error bounds Se ii} Let $F$ be coercive with respect to $\Omega$. 
Suppose that there exist  $\gamma> 1$ and
$\alpha>0$ such that \eqref{weaksharp-gamma-stronger} holds.
Then there exists $\tau>0$ such that, for all $\epsilon>0$,
  \begin{equation}
   d(x_\epsilon, S_0)^{\gamma-1}\leq \tau\epsilon~~~\forall~~ x_\epsilon\in S_\epsilon.
  \end{equation}
In particular, $S_\epsilon$ is bounded for each $\epsilon>0$. 
 \end{enumerate}

 \end{thm}
\begin{proof}
We begin with some general observations.  
Let $\epsilon>0$. For any $x_\epsilon \in
S_\epsilon$  
 \begin{equation*}
  \DRLip{F(x_\epsilon)+\epsilon \nabla\varphi(x_\epsilon)}{y-x_\epsilon}\geq 0\quad
\forall y\in \Omega.
 \end{equation*}
Rearranging yields
 \begin{equation*}
  \DRLip{F(x_\epsilon)}{x_\epsilon-y}\leq \epsilon\DRLip{
\nabla\varphi(x_\epsilon)}{y-x_\epsilon}\quad\forall y\in \Omega.
 \end{equation*}
 Since $F$ is monotone and $\varphi$ is convex, for all $y\in\Omega$, it 
 holds that
 \[
 \DRLip{F(y)}{x_\epsilon-y}\leq \DRLip{F(x_\epsilon)}{x_\epsilon-y}\leq
\epsilon  \DRLip{\nabla\varphi(x_\epsilon)}{y-x_\epsilon}\leq \epsilon
(\varphi(y)-\varphi(x_\epsilon)).
\]
 In particular, for $\xbar_\epsilon:=P_{S_0}(x_\epsilon) \in \Omega$ (the projection is nonempty
by Assumption \ref{hyp:GVIvarphi}\eqref{hyp:1}), we have 
  \begin{equation}\label{e:monotoneandconvex}
\DRLip{F(\xbar_\epsilon)}{x_\epsilon-\xbar_\epsilon}\leq \epsilon
(\varphi(\xbar_\epsilon)-\varphi(x_\epsilon)).
 \end{equation}

\eqref{t:error bounds Se i}. 
The inequality, \eqref{weaksharpimplication}, together with \eqref{e:monotoneandconvex} 
immediately yields
   \begin{equation*}
\alpha~d(x_\epsilon, S_0)\leq \epsilon
(\varphi(\xbar_\epsilon)-\varphi(x_\epsilon))
 \end{equation*}
as claimed.

\eqref{t:error bounds Se ii}. Inequalities \eqref{e:monotoneandconvex} and \eqref{weaksharp-gamma-stronger} yield
 \begin{equation}\label{e:dist_ineq1}
\alpha~d(x_\epsilon, S_0)^\gamma \leq \epsilon
(\varphi(\xbar_\epsilon)-\varphi(x_\epsilon)).
 \end{equation}
Since $\varphi$ is convex real
valued,
 \begin{equation*}
  \DRLip{\nabla\varphi(\xbar_\epsilon)}{x_\epsilon-\xbar_\epsilon}\leq
(\varphi(x_\epsilon)-\varphi(\xbar_\epsilon)).
 \end{equation*}
Using Cauchy-Schwarz inequality
 \begin{equation*}
 -\|\nabla\varphi(\xbar_\epsilon)\|_2\|x_\epsilon-\xbar_\epsilon\|_2 \leq
\DRLip{\nabla\varphi(\xbar_\epsilon)}{x_\epsilon-\xbar_\epsilon}\leq
\varphi(x_\epsilon)-\varphi(\xbar_\epsilon).
 \end{equation*}
 Since $\xbar_\epsilon:=P_{S_0}(x_\epsilon)$. this implies that
  \begin{equation}\label{e:dist_ineq2}
 -d(x_\epsilon, S_0)\|\nabla\varphi(\xbar_\epsilon)\|_2\leq
\varphi(x_\epsilon)-\varphi(\xbar_\epsilon).
 \end{equation}
Combining \eqref{e:dist_ineq2} and \eqref{e:dist_ineq1} yields
  \begin{equation*}
\alpha d(x_\epsilon, S_0)^{\gamma-1} \leq 
\epsilon\|\nabla\varphi(\xbar_\epsilon)\|_2.
 \end{equation*}
Now, for $F$ and $\Omega$ satisfying Assumption \ref{hyp:F} with 
$F$ coercive on $\Omega$, the solution set $S_0$ is bounded (Proposition \ref{t:existence}\eqref{t:existence i}). 
Moreover, by Assumption \ref{hyp:GVIepsilon},
$\varphi$ is convex and, by assumption differentiable, on $\Omega$, and hence convex and differentiable on $S_0$.   Consequently,  
$\nabla\varphi$ is bounded uniformly on the compact set $S_0$.   
Hence the proof
follows 
 with $\tau=\alpha^{-1}M$ where $M$ is the uniform bound for
$\|\nabla\varphi(\cdot)\|$.
\end{proof}

\section{Numerical Illustration and Conclusion}
We illustrate the theory explored in the previous sections and indicate 
directions for future investigation with numerical experiments on the following 
simple example.
\subsection{Best Approximation}
\begin{eg}\label{eg:ba}
 Let 
 \[
  \Omega\equiv\left\{\left.x = (x_1,x_2,x_3)^T\in\Rbb^3~\right|~\langle n,x\rangle = 
  -1, ~ x_1\leq 1, n=(0,1,1)^T\right\}
 \]
and define $F(x)\equiv x-P_{C}(x)$ where $C\equiv \Rbb^3_++(0,-1/4,1/4)$.    
We compare two regularizing functions, $\varphi_1(x)\equiv \|x\|_1$
and $\varphi_2(x)\equiv\tfrac12 \|x\|_2^2$ (shifted Tikhonov)  for the approaches 
to solving \eqref{e:VI} explored separately in Section \ref{s:reg G} and 
Section \ref{s:reg VI}, namely by solving \eqref{e:PGe} and \eqref{e:GVIe}
respectively.  

For this problem  we know the following.
\begin{itemize}
 \item $S_0=\{(x, -\tfrac34, -\tfrac14)~|~x\in[0,1]\}$.
 \item $S_{\varphi_j} = \{(0, -\tfrac34, -\tfrac14)\}, ~(j=1,2)$.
 \item For the regularizer $\varphi_1$, $\SGeps = \{(0, -\tfrac34, -\tfrac14)\}$ for all $\epsilon>0$.\\
{\em Proof sketch}. 
The nearest points in $\Omega$ to 
 the point $x_0=(0,0,0)$ with respect to the $\ell^1$ norm are points on the line segment 
 $y(t)=t(0,-1,1)+(0,0,-1)$
 for $t\in[0,1]$, and this line segment intersects $S_0$ at the point 
 $(0, -\tfrac34, -\tfrac14)$, where, we know, $G(x)$ attains
 its minimum.  \hfill $\triangle$
 \item For the regularizer $\varphi_2$, the shifted 
 Tikhonov regularizer, $\SGeps\cap S_0=\emptyset$ for all $\epsilon>0$.\\
{\em Proof sketch}.  The global minimum of 
 $\varphi_2$ on $\Omega$, namely the point $(0,-1/2, -1/2)$, 
 does not coincide with those of $G(x)$ ($S_0$). 
 Moreover, $\varphi_2$ is strictly convex on $\Omega$, so the global minimum of the 
 sum cannot be on $S_0$ for any value of $\epsilon$. \hfill $\triangle$
 \item For the regularization $\varphi_1$, 
  $S_0\cap S_\epsilon =  \{(0, -\tfrac34, -\tfrac14)\}$ and, in fact 
 $S_\epsilon =  \{(0, -\tfrac34, -\tfrac14)\}$ for all $\epsilon>0$.  \\
{\em Proof sketch}. 
Again,  because the nearest point in $\Omega$ to 
 the origin with respect to the $\ell^1$ norm are all points $y(t)=t(0,-1,1)+(0,0,-1)$
 for $t\in[0,1]$, and this line segment intersects $S_0$ at the point $(0, -\tfrac34, -\tfrac14)$, by 
 Theorem \ref{t:solution sets}\eqref{t:solution sets 3} the claim follows.  \hfill $\triangle$
 \item The regularization $\varphi_2$ in \eqref{e:GVIe} is not exact.\\
{\em Proof sketch}. 
 For this regularization, a short calculation 
 shows that, for all $\epsilon>0$, $S_\epsilon$ is a unique point on the line segment  $y(t)=t(0,-1,1)+(0,0,-1)$
 for $t$ in the open interval $(\tfrac14,\tfrac12)$.  This interval does not intersect $S_0$, that is, 
$S_0\cap S_\epsilon=\emptyset$ for all $\epsilon>0$. \hfill $\triangle$
 \item $\lambda\in[0,+\infty)$ are Lagrange multipliers of $\gvi(\partial\varphi_1, S_0)$.  
\end{itemize}

The numerical results reported in Table \ref{tbl:1} were generated 
from the same initial point, $(1,-2, 1)$. 
\begin{table}[htbp]
 \caption{Comparison of optimization models \eqref{e:GVIe} and \eqref{e:PGe}
 with different regularizations ($\ell^1$ or $\ell^2$) and different regularization 
 parameters $\epsilon$.}\label{tbl:1}
 \vspace{5pt}
 \resizebox{\textwidth}{!}{
  \begin{tabular}{lcccc} \hline
   Problem                   & Iteration & CPU (sec)     & Distance to Solution & Distance to $S_0$     \\ \hline
\eqref{e:PGe}, $\varphi=\ell^1$, $\epsilon=0.5$ 	  &  20& 6.889 	& $1.357\times 10^{-5}$  &$1.357\times 10^{-5}$ 	 \\
\eqref{e:PGe}, $\varphi=\ell^1$, $\epsilon=0.1$ 	  &  22 &6.440	&  $2.264\times 10^{-9}$ &$2.264\times 10^{-9}$\\
\eqref{e:PGe}, $\varphi=\ell^1$, $\epsilon=0.01$	  &  32&6.353 	& $8.412\times 10^{-10}$ &$8.412\times 10^{-10}$  \\
\eqref{e:PGe}, $\varphi=\ell^1$, $\epsilon=0.005$	  &  37& 8.552 	& $2.660\times 10^{-9}$ &$2.660\times 10^{-9}$ \\
\eqref{e:PGe}, $\varphi=\ell^1$, $\epsilon=0.0001$ &  29 &6.903	&  $3.285\times 10^{-9}$ &$3.285\times 10^{-9}$\\
\eqref{e:PGe}, $\varphi=\ell^2$, $\epsilon=0.5$ 	   &  8& 8.650 	& $1.768\times 10^{-1}$  &$1.768\times 10^{-1}$ 	 \\
\eqref{e:PGe}, $\varphi=\ell^2$, $\epsilon=0.1$ 	   &  10& 8.670 	& $5.893\times 10^{-2}$  &$5.893\times 10^{-2}$\\
\eqref{e:PGe}, $\varphi=\ell^2$, $\epsilon=0.01$	   &  20 &11.68	& $6.931\times 10^{-3}$ &$6.931\times 10^{-3}$  \\
\eqref{e:PGe}, $\varphi=\ell^2$, $\epsilon=0.005$   &  19& 12.85 	&  $3.500\times 10^{-3}$& $3.500\times 10^{-3}$\\
\eqref{e:PGe}, $\varphi=\ell^2$, $\epsilon=0.0001$ &  29 & 23.35	& $6.078\times 10^{-5}$  &$1.83\times 10^{-2}$\\
\eqref{e:GVIe}, $\varphi=\ell^1$, $\epsilon=0.5$ 	  &  71&.1099 	& $7.930\times 10^{-10}$ &$.7930\times 10^{-9}$	 \\
\eqref{e:GVIe}, $\varphi=\ell^1$, $\epsilon=0.1$ 	  &  74&.0860 	&  $.0248$ &$.9277\times 10^{-9}$\\
\eqref{e:GVIe}, $\varphi=\ell^1$, $\epsilon=0.01$	  &  75& .0868 	&  $.8125$ &$.9653\times 10^{-9}$ \\
\eqref{e:GVIe}, $\varphi=\ell^1$, $\epsilon=0.005$  &  75&.0859 	& $.9063$ &$.9876\times 10^{-9}$  \\
\eqref{e:GVIe}, $\varphi=\ell^1$, $\epsilon=0.0001$ &  76&.0883 	& $.9981$& $.7570\times 10^{-9}$\\
\eqref{e:GVIe}, $\varphi=\ell^2$, $\epsilon=0.5$ 	  &  66&.1018 	& $1.571\times 10^{-9}$&$1.179\times 10^{-1}$  	 \\
\eqref{e:GVIe}, $\varphi=\ell^2$, $\epsilon=0.1$ 	  &  281&.3143	& $1.725\times 10^{-9}$  &$3.21\times 10^{-2}$\\
\eqref{e:GVIe}, $\varphi=\ell^2$, $\epsilon=0.01$	  &  1910& 2.153 & $1.763\times 10^{-9}$  &$3.5\times 10^{-3}$\\
\eqref{e:GVIe}, $\varphi=\ell^2$, $\epsilon=0.005$  &  3267& 3.697 & $1.766\times 10^{-9}$&$1.8\times 10^{-3}$ \\
\eqref{e:GVIe}, $\varphi=\ell^2$, $\epsilon=0.0001$ &  6932& 7.621 & $1.768\times 10^{-9}$&$3.54\times 10^{-5}$\\
  \hline
  \end{tabular}
 }\\
\end{table}
\hfill $\Box$
\end{eg}

Example \eqref{eg:ba} has been purposely designed for simplicity - there are clearly 
other ways to solve the variational inequality.  Recognizing that the problem is one 
of finding nearest points on the half-plane $\Omega$ to the shifted orthant, simple 
alternating projections would converge to an exact solution finitely, without 
recourse to regularization.  Our purpose, however, is not to explore 
efficient algorithms for solving this particular problem, but rather to illustrate 
the theory of (exact) regularization and to underscore the possible 
advantages of different modelling approaches.  

The optimization problem \eqref{e:PGe} was solved using Matlab's {\em fmincon} with 
an interior point solver.  Evaluation of the dual gap function $G$ given by \eqref{e:dual gap func}
also involves solving an optimization problem.  For the problem in Example \eqref{eg:ba}
this has an explicit representation, but in general this will not be the case.  We therefore
evaluate the dual gap function numerically so that the experimental results will accurately
simulate a practical implementation.  

There are a variety of ways to solve \eqref{e:GVIe}.  We briefly describe an approach here
where the error bounds derived in Section \ref{s:error bounds S_eps} are put to use. This is
an Armijo
descent type algorithm by Li and Ng applicable for a Lipschitz, coercive mapping \cite{LiNg09}.
Hence, the analysis of Li and Ng  
applies to the problem $VI(T_{\epsilon}, \Omega)$, with appropriate assumptions on $F$ and $\varphi$ 
under which $T_\epsilon$ is Lipschitz and coercive on $\Omega$. This method uses the descent
direction
\begin{equation}
\label{direction}
d^k:=\begin{cases}
    y^\varphi_{{\alpha_k}, \epsilon_k}(x)-y^\varphi_{{\beta_k}, \epsilon_k}(x), & \text{if
$c_k\|x-y^\varphi_{{\alpha_k}, \epsilon_k}(x)\|\leq \|y^\varphi_{{\alpha_k},
\epsilon_k}(x)-y^\varphi_{{\beta_k},
\epsilon_k}(x)\|$},\\
    y_{{\alpha_k}, \epsilon_k}(x)-x, & \text{otherwise};
  \end{cases}
\end{equation}
where $y^\varphi_{{\alpha_k}, \epsilon_k}(x)=P_\Omega[x-\frac{1}{{\alpha_k}}T_{\epsilon_k}( x)]$, the point where the supremum in $\theta_{{\alpha_k}}(x,\epsilon_k\varphi)$ is attained
and $c_k$ is chosen to satisfy
\[
 c_k\leq \min\left\{1, \frac{{\beta_k}-{\alpha_k}}{2(L^\theta_k+{\beta_k})}\right\}.
 \]
Here $L^\theta_k$ is the Lipschitz constant of $\theta_{{\alpha_k}{\beta_k}}(.,\epsilon_k\varphi)$ 
on $Lev^\theta_k= \{x:\theta_{{\alpha_k}{\beta_k}}(x;\epsilon_k\varphi)\leq
\theta_{{\alpha_k}{\beta_k}}(x^{k,0};\epsilon_k\varphi)\}$, where $x^{k,0}$ is the chosen initial point 
for the inner iteration for each $k$.
It has a step rule that finds  the smallest non-negative integer $m$ such that,
\[
\sqrt{ \theta_{{\alpha_k}{\beta_k}}(x^{k,j}+\gamma_k^m d^{k,j}, \epsilon_k\varphi)}-
\sqrt{\theta_{{\alpha_k}{\beta_k}}(x^{k,j}, \epsilon_k\varphi)}\leq
-\frac{\delta_k}{4}\gamma_k^m\|d^k\|,
\]
and updates $x^{k,j}$ as
\[
x^{k, j+1}=x^{k,j}+t_{k,j} d^{k,j}~~\mbox{where}~~  t_{k,j}= \gamma_k^m, \gamma_k\in (0,1),
\]
where the constant $\delta_k$ chosen for a strongly monotone map $T_{\epsilon_k}$ satisfies \cite[Remark 4.3]{LiNg09}
\[
 \delta_k\leq \min
 \left\{\frac{1}{2}\sqrt{\frac{{\beta_k}-{\alpha_k}}{2}}, \frac{\sqrt{2}c_k
\mu^k_{{\alpha_k}{\beta_k}}}{\sqrt{{\beta_k}-{\alpha_k}}}\right\},
\] and $\mu^k_{{\alpha_k}{\beta_k}}$ is the modulus of strong monotonicity of $T_{\epsilon_k}$. 
We use an estimate for $L^\theta_k$ and a step size $\gamma_k=.9$ for all $k$. We note that,
since $P_{C}$ for $C\equiv\mathbb{R}^3_+  + ( 0, -\frac{1}{4}, \frac{1}{4}) $ is nonexpansive, 
the Lipschitz constant of $F$ defined by $I-P_C$ where $I$ is the 
identity mapping, is $L=2$.\\\\
\begin{itemize}
  \item {\em The regularizer $\varphi_2 = \tfrac12\|\cdot\|_2^2$}: 
 Since the modulus of strong monotonicity of  $\nabla\varphi_2$ is $\rho=1$, $T_{\epsilon_k}$ is strongly monotone
 with modulus of strong monotonicity $\epsilon_k$ and hence $\delta_k$ in this case is chosen to satisfy
\[
 \delta_k\leq \min
 \left\{\frac{1}{2}\sqrt{\frac{{\beta_k}-{\alpha_k}}{2}}, \frac{\sqrt{2}c_k
\epsilon_k}{\sqrt{{\beta_k}-{\alpha_k}}}\right\}.
\] Note that $T_{\epsilon_k}$ is Lipschitz on $\Omega$. Since $T_{\epsilon_k}$
is strongly monotone, it is coercive too on $\Omega$
 \cite[Remark 2.1]{LiNg09}. So, we can 
 apply the method also calculate the error bound $p_k$ in 
\eqref{alternate_stoppingrule}. Noting that the Lipschitz constant of $\nabla\varphi_2$ is $M=1$,
choosing $\alpha_k=1$ and $\beta_k=2$ for all $k$, the constant $L_k$ in \eqref{e:ineq_dist} is 
calculated as $L_k= \frac{\beta_k+L+\epsilon_k M}{\epsilon_k\rho
}\sqrt{\frac{2}{(\beta_k-\alpha_k)}}=\frac{4+\epsilon_k }{\epsilon_k
}\sqrt{2}$. We choose $\tau_k= 10^{-8}/\epsilon$ in the tolerance $p_k=\frac{\tau_k^2}{L_k^2}$ in 
\eqref{alternate_stoppingrule}. 
  \item {\em The regularizer $\varphi_1 = \|\cdot\|_1$:} There is no available theory.  We include this 
  experiment to indicate the potential for this approach, and, hopefully, to inspire more 
  research to explain these results.  
  \end{itemize}
\medskip

\begin{remark}
A few trends from Table \ref{tbl:1} are worth noting before we conclude.  First, while 
exact regularization of \eqref{e:PGe} converges to a solution of the unregularized variational inequality, it 
requires more iterations than (inexact) regularization via the $\ell^2$ norm.  Nevertheless, the 
per iteration computational cost, as shown by the CPU times, indicates that the nonsmooth 
regularization is still more efficient.  This could 
be due to our solution technique for the smooth regularization.  If a 
more efficient method for smooth regularization were available, an iteratively regularized
problem, along the lines of Algorithm \ref{alg},  could be a reasonable strategy.  Such a strategy is  
made possible by the error bound established in Theorem \ref{t:theorem-errorbound-ereg}.  For direct 
regularization following model \eqref{e:GVIe}, we 
have implemented Algorithm \ref{alg} with stopping criteria given by the error bounds established in 
Theorem \ref{t:errorbound}.  This performs as expected for smooth regularization.  The distance 
to the solution to the regularized problem is reported according to the upper bound established in 
Theorem \ref{t:errorbound}.  What is not covered by the theory 
developed here are the results of our solution to model \eqref{e:GVIe} with the nonsmooth 
regularization $\varphi = \|\cdot\|_1$.  Since we know the answer, we monitored the 
  distance of the iterates to the solution of the regularized and unregularized problems. 
  The gap functions and stopping criteria developed for the case of smooth regularization 
 was not useful or even remotely informative regarding the progress of the iterates.  Nevertheless, 
 the direction choice and backtracking procedures appear to function well for this example.  The 
 algorithm appears to move quickly to the set $S_0$, but then cannot make further progress to the solution to the regularized 
 problem, which consists of a single element from $S_0$. Finally, we note that, as indicated by the tabulated 
 CPU times, the iteration counts
 should only be used as an indication of the {\em relative} computational complexity.  One iteration 
 of the method of Li and Ng for solving \eqref{e:GVIe} is a tiny fraction of the computational cost of 
 one iteration of our approach to solving \eqref{e:PGe}.  
\end{remark}

\subsection{Conclusion}
Our inspiration for this study was the theory of exact regularization in optimization developed in 
\cite{FriedlanderTseng07}.  As with optimization, exact regularization for variational inequalities is 
closely related to the existence of Lagrange multipliers for a related optimization problem, 
namely \eqref{e:PphiG} (Definition \ref{d:Lagrange VI}.  We have found that the dual gap function 
defined by \eqref{e:dual gap func} plays a central role here.  The dual gap function is 
difficult to work with in practice since it is itself the supremum of a nonlinear objective.  
We determined that, even in the absence of exact regularization,  
it is possible to establish error bounds for both model approaches to the true solution set 
introducing
the notion of {\em weak-sharp minimum of degree $\gamma$} defined by 
\eqref{weaksharpVIG-gamma} and 
\eqref{weaksharp-gamma-stronger} respectively.  Two avenues for 
further exploration present themselves.  One direction is an investigation of efficient numerical 
strategies based approximations to the dual gap functional $G$.  The second direction 
is an investigation 
of generalized variational inequalities to accommodate nonsmooth, set-valued regularization for 
regularized variational inequalities of the form \eqref{e:GVIe}.   Both of these topics are 
formidable challenges. 

\subsection{Acknowledgments}
The research of CC and DRL was supported by the Deutsche Forschungsgemeinschaft/German Research Foundation grant SFB755-TPA4.

\end{document}